\documentclass[a4paper, 11pt]{amsart}
\usepackage[latin1]{inputenc}
\usepackage[T1]{fontenc}
\usepackage[francais]{babel}
\usepackage{amssymb}
\usepackage{amsmath}
\usepackage{amsthm}
\usepackage{amscd}
\usepackage{amsfonts}
\usepackage{stmaryrd}
\usepackage{pb-diagram}
\usepackage{epic,eepic,epsfig}
\usepackage{a4wide}
\usepackage{enumerate}
\usepackage{nextpage}
\usepackage{fancyhdr}
\newtheorem{prop}{Proposition}[section]
\newtheorem{cor}[prop]{Corollaire}

\newtheorem{lem}[prop]{Lemme}
\newtheorem{thm}[prop]{Th{\'e}or{\`e}me}

\newtheorem{defin}[prop]{D\'efinition}

\newtheorem{clef}[prop]{\'Egalit\'e Clef}

\newtheorem{propdef}[prop]{Proposition-D\'efinition}
\newtheorem{theorem}[prop]{Th{\'e}or{\`e}me}
\newtheorem{corol}[prop]{Corollaire}

\newtheorem{Conjecture}[prop]{Conjecture}
\theoremstyle{definition}
\newtheorem{Remarque}[prop]{Remarque}
\newtheorem{Exemple}[prop]{Exemple}

\newcommand{\Jac} {\mathop{\mathrm{Jac}}}
\newcommand{\Ima} {\mathop{\mathrm{Im}}}
\newcommand{\Tr} {\mathop{\mathrm{Tr}}}
\newcommand{\Trinf} {\mathop{\mathrm{Tr}_{\infty}}}

\newcommand{\sinf} {\mathop{\mathrm{s}_{\infty}}}
\newcommand{\disc} {\mathop{\mathrm{disc}}}

\newcommand{\ordp} {\mathop{\mathrm{ord}_{p}}}
\newcommand{\ordv} {\mathop{\mathrm{ord}_{v}}}
\newcommand{\Ree} {\mathop{\mathrm{Re}}}
\newcommand{\Card} {\mathop{\mathrm{Card}}}
\newcommand{\Sp} {\mathop{\mathrm{Sp}}}
\newcommand{\Sym} {\mathop{\mathrm{Sym}}}
\newcommand{\divi} {\mathop{\mathrm{div}}}
\newcommand{\pgcd} {\mathop{\mathrm{pgcd}}}

\newcommand{\supp} {\mathop{\mathrm{supp}}}
\newcommand{\Divi} {\mathop{\mathrm{Div}}}

\newcommand{\cl} {\mathop{\mathrm{cl}}}
\newcommand{\Nk} {\mathop{\mathrm{N}_{k/\mathbb{Q}}}}
\newcommand{\Nkprime} {\mathop{\mathrm{N}_{k'/\mathbb{Q}}}}
\newcommand{\rang} {\mathop{\mathrm{rang}}}
\newcommand{\dime} {\mathop{\mathrm{dim}}}

\newcommand{\hF} {\mathop{h_{\mathrm{F}}}}
\newcommand{\hFprime} {\mathop{h_{\mathrm{F}}'}}
\newcommand{\hstprime} {\mathop{h_{\mathrm{st}}'}}
\newcommand{\hst} {\mathop{h_{\mathrm{st}}}}

\newcommand{\Gal} {\mathop{\mathrm{Gal}}}

\begin{document}

\title{Minoration de la hauteur de N\'eron-Tate sur les surfaces ab\'eliennes}
\author{Fabien Pazuki}

\maketitle 

\begin{center}
Manuscripta Mathematica 229-0593
\end{center}

\vspace{0.5cm}

\begin{centering}
{\small{\textsc{R\'esum\'e} : On obtient dans le pr\'esent texte des r\'esultats en direction d'une conjecture de Lang et Silverman de minoration de la hauteur canonique sur les vari\'et\'es ab\'eliennes de dimension 2 sur un corps de nombres. La m\'ethode utilis\'ee est une d\'ecomposition en hauteurs locales. On d\'eduit en corollaire une borne uniforme sur la torsion de familles de surfaces ab\'eliennes et une borne uniforme sur le nombre de points rationnels de familles de courbes de genre 2.}}
\end{centering}

\vspace{0.5cm}

\begin{centering}
{\small{\textsc{Abstract} : This paper contains results concerning a conjecture made by Lang and Silverman, predicting a lower bound for the canonical height on abelian varieties of dimension 2 over number fields. The method used here is a local height decomposition. We derive as corollaries uniform bounds on the number of torsion points on families of abelian surfaces and on the number of rational points on families of genus 2 curves.}}
\end{centering}

\vspace{0.5cm}

{\flushleft
}

\vspace{0.5cm}

\thispagestyle{empty}

\section{La conjecture de Lang et Silverman}

\subsection{Pr\'esentation}
Soit $k$ un corps de nombres de degr\'e $d$ sur $\mathbb{Q}$. On va s'int\'eresser \`a une question figurant dans le livre de S. Lang \cite{Lan} page 92 et qui concerne la minoration de la hauteur de N{\'e}ron-Tate d'un point rationnel d'ordre infini sur une courbe elliptique. Cette question a \'et\'e la source d'un grand nombre de travaux et de g\'en\'eralisations en g\'eom\'etrie diophantienne. On peut la formuler de la mani\`ere suivante :

\begin{Conjecture}(Lang)
Pour tout corps de nombres $k$, il existe une constante positive $c(k)$
telle que pour toute courbe elliptique $E$ d{\'e}finie sur $k$ et tout
point $P$ d'ordre infini de $E(k)$ on ait :
$$\widehat{h}(P) \geq c(k)\, \max\Big\{\log \Nk(\Delta_{E}),h(j_{E})\Big\}, $$
o{\`u} $\widehat{h}(.)$ est la hauteur de N{\'e}ron-Tate sur $E$,
$\Nk(\Delta_{E})$ la norme du disciminant minimal de la
courbe $E$ et $h(j_{E})$ la hauteur de Weil logarithmique et absolue
de l'invariant modulaire $j_{E}$ de la courbe $E$.
\end{Conjecture}

On trouve des r\'esultats en direction de cet \'enonc\'e dans les travaux de J. Silverman \cite{Sil4} et \cite{Sil3}, M. Hindry et J. Silverman dans \cite{HiSi3} et de S. David dans \cite{Dav3}. Citons aussi M. Krir \cite{Krir} et C. Petsche \cite{Pet}. M. Hindry et J. Silverman obtiennent dans \cite{HiSi3}, corollaire 4.2 (ii) de leur th\'eor\`eme 4.1 (page 430 et 431), le r\'esultat suivant :

\begin{theorem}(Hindry, Silverman) \label{elliptique}
Soit $k$ un corps de nombres de degr\'e $d$. Soit $E/k$ une courbe elliptique de
disciminant minimal $\Delta_{E}$ et de conducteur $F_{E}$. On note $\sigma_{E}$ le
quotient de Szpiro d\'efini par $\sigma_{E} = \log \Nk(\Delta_{E})/\log \Nk(F_{E})$. Alors
pour tout point $P\in{E(k)}$ d'ordre infini on a la minoration :
\[\widehat{h}(P)\geq (20\sigma_{E})^{-8d}10^{-4\sigma_{E}}\frac{1}{12}\max\Big\{\log \Nk(\Delta_{E}),h(j_{E})
\Big\}.\]
\end{theorem}
Une conjecture de Szpiro, \'equivalente \`a une forme de la conjecture ABC, affirme que $\sigma_{E}$ est uniform\'ement born\'e et entra\^ine donc la conjecture de Lang via ce th\'eor\`eme de Hindry et Silverman. La conjecture sur les courbes elliptiques a ensuite \'et\'e g{\'e}n{\'e}ralis{\'e}e aux vari{\'e}t{\'e}s ab{\'e}liennes de dimension sup{\'e}rieure par J. Silverman dans \cite{Sil3} page 396  :

\begin{Conjecture}(Lang, Silverman)
Soit $g \geq 1$. Pour tout corps de nombres $k$, il existe une
constante strictement positive $c(k,g)$ telle que pour toute vari{\'e}t{\'e} ab{\'e}lienne
$A/k$ de dimension $g$, pour tout diviseur ample et sym\'etrique $\mathcal{D}\in{\Divi(A)}$ et
tout point $P\in{A(k)}$ tel que $\mathbb{Z}\!\cdot\! P=\{mP|m\in{\mathbb{Z}}\}$ est Zariski-dense on ait :
$$\widehat{h}_{A,\mathcal{D}}(P) \geq c(k,g)\, \max\Big\{1,\hF(A/k)\Big\}, $$
o{\`u} $\widehat{h}_{A,\mathcal{D}}(.)$ est la hauteur de N{\'e}ron-Tate sur $A$ associ{\'e}e au
diviseur $\mathcal{D}$ et $\hF(A/k)$ est la hauteur de Faltings (relative) de la vari{\'e}t{\'e}
ab{\'e}lienne $A$.
\end{Conjecture}

\begin{Remarque}
Il y a plusieurs notions de hauteur d'une vari\'et\'e ab\'elienne $A$. L'\'enonc\'e de cette conjecture est plus fin avec la hauteur de Faltings \emph{relative} $\hF(A/k)$ comme minorant qu'avec la hauteur de Faltings \emph{stable} $\hst(A)$. Rappelons de plus que la hauteur de Faltings stable est comparable \`a une hauteur modulaire (voir \cite{Falt}, \cite{DavPhi} ou \cite{Paz2}), comme par exemple la hauteur th\^eta $h_{\Theta}(A)$, v\'erifiant $|\hF(A)-2h_{\Theta}(A)|\ll\log h_{\Theta}(A)$, o\`u la constante implicite d\'epend seulement de la dimension $g$ et d'un niveau de torsion $r$ fix\'e (on prendra en g\'en\'eral $r=4$). Dans ce texte on utilisera aussi la hauteur de Faltings modif\'ee relative (d\'efinie plus bas par la formule (\ref{hauteur Faltings modifi\'ee}), voir aussi \cite{Paz2}) not\'ee $\hFprime(A/k)$, et son avatar stable not\'ee $\hstprime(A)$, qui v\'erifie $|\hstprime(A)-2h_{\Theta}(A)|\ll 1$ (o\`u la constante implicite d\'epend seulement de la dimension $g$ et d'un niveau de torsion $r$ fix\'e).
\end{Remarque}

S. David a propos\'e dans \cite{Dav2} une preuve partielle de cette conjecture
g{\'e}n{\'e}ralis{\'e}e, preuve bas\'ee sur un raisonnement de type transcendance : il donne une borne inf{\'e}rieure pouvant tendre vers l'infini avec la hauteur th\^eta de la vari{\'e}t{\'e}. 

\begin{theorem}(David)
Soient $g\geq 1$ un entier, $k$ un corps de nombres, $v$ une place archim\'edienne, $(A,\mathcal{D})/k$ une vari\'et\'e ab\'elienne principalement polaris\'ee de dimension $g$ et $\tau_{v}$ une matrice du domaine de Siegel (voir paragraphe \ref{domaine archim}) telle que $A(\bar{k}_{v})\cong \mathbb{C}^{g}/\mathbb{Z}^{g}+\tau_{v}\mathbb{Z}^{g}$. On note $\parallel\Ima\tau_{v}\parallel=\max_{i,j} |\Ima\tau_{v,ij}|$. Posons~: $\rho(A)=h_{\Theta}(A)/\parallel\Ima\tau_{v}\parallel.$

Alors il existe une constante $c_{1}(k,g)>0$ telle que, tout point $P\in{A(k)}$ v\'erifiant que $\mathbb{Z}.P$ est Zariski-dense, on a~:
\[
 \widehat{h}_{A,\mathcal{D}}(P)\geq c_{1}(k,g)\rho(A)^{-4g-2}\Big(\log\rho(A)\Big)^{-4g-1}\,h_{\Theta}(A).
\]
\end{theorem}

Cet \'enonc\'e implique donc l'in\'egalit\'e cherch\'ee (sur un corps o\`u la r\'eduction est semi-stable) pour les familles de vari\'et\'es ab\'eliennes v\'erifiant $\rho(A)$ born\'e. D. Masser utilise d'ailleurs ces r{\'e}sultats dans \cite{Masser} pour exhiber une famille de vari\'et\'es ab\'eliennes simples
avec $\rho$ born\'e uniform\'ement.

En application, on donnera des r\'esultats en direction de deux conjectures classiques dont on rappelle les \'enonc\'es ici :

\begin{Conjecture}(de torsion forte)
Soient $k$ un corps de nombres de degr\'e $d$ et $g\geq 1$ un entier. Alors il existe une constante $c(d,g)>0$ ne d\'ependant que de $d$ et $g$ telle que pour toute vari\'et\'e ab\'elienne $A$ de dimension $g$ d\'efinie sur $k$ on a :
$$\Card A(k)_{\mathrm{tors}}\leq c(d,g).$$
\end{Conjecture}

\begin{Conjecture}(points rationnels)
Soient $k$ un corps de nombres de degr\'e $d$ et $g\geq 2$ un entier. Alors il existe une constante $c(k,g)>0$ ne d\'ependant que de $k$ et $g$ telle que pour toute courbe $C$ de genre $g$ d\'efinie sur $k$ on a :
$$\Card C(k)\leq c(k,g)^{\mathrm{rang}_{k}(\Jac(C))+1},$$
o\`u $\Jac(C)$ d\'esigne la vari\'et\'e jacobienne de $C$.
\end{Conjecture}

\subsection{R\'esultats}

Une vari\'et\'e ab\'elienne principalement polaris\'ee de dimension 2 est isomorphe ou bien \`a une jacobienne d'une courbe $C$ de genre 2 polaris\'ee par le diviseur $\Theta= C$, ou bien \`a un produit de courbes elliptiques $E_{1}\!\times\! E_{2}$, polaris\'e par $\Theta= E_{1}\!\times\!\{O\}\!+\!\{O\}\!\times\! E_{2}$.

On obtient dans cet article un th\'eor\`eme de minoration de la hauteur de N\'eron-Tate associ\'ee au diviseur $\Theta$ en utilisant une technique de d\'ecomposition en hauteurs locales l\'eg\`erement modifi\'ees. En effet ces hauteurs locales sont d\'efinies \`a une constante additive pr\`es, il y a donc plusieurs mani\`eres de normaliser ces fonctions. On met en place une \'etude des diff\'erences de hauteurs locales (c'est une mani\`ere d\'etourn\'ee de fixer une normalisation) grâce \`a une propri\'et\'e cruciale des points de $3$-torsion en dimension 2. 

La m\'ethode de d\'ecomposition locale et l'\'etude des s\'eries th\^eta associ\'ees fait appara\^itre une condition n\'ecessaire dans l'espace de modules des vari\'et\'es ab\'eliennes principalement polaris\'ees de dimension 2. Le ph\'enom\`ene de rupture d'une vari\'et\'e ab\'elienne simple en produit de courbes elliptiques entra\^ine une explosion des composantes locales, tant au niveau de la minoration de la hauteur de N\'eron-Tate que de la majoration de la hauteur de Faltings. On va donc introduire une quantit\'e appel\'ee \textit{simplicit\'e archim\'edienne} charg\'ee de mesurer la distance au produit de courbes elliptiques.

Dans tout le texte on note $M_{k}$ l'ensemble de ses places (deux \`a deux non \'equivalentes), $M_{k}^{\infty}$ l'ensemble de ses places archim\'ediennes et $M_{k}^{0}$ l'ensemble de ses places finies. Pour toute place $v$ de $k$ on note $k_{v}$ le compl\'et\'e de $k$ pour la valuation $|.|_{v}$ associ\'ee o\`u on normalise $|p|_{v}=p^{-1}$ pour toute place finie $v$ au-dessus d'un nombre premier $p$. On pose $d_{v}=[k_{v}:\mathbb{Q}_{v}]$ et $n_{v}=d_{v}/d$. Pour une surface ab\'elienne principalement polaris\'ee $A/k$ avec $k$ un corps de nombres et $v$ une place infinie, on peut uniformiser les points complexes $A(\bar{k}_{v})\cong \mathbb{C}^{2}/\mathbb{Z}^{2}+\tau_{v}\mathbb{Z}^{2}$ avec $\tau_{v}=\left[ \begin{array}{cc}
\tau_{1,v} & \tau_{12,v}   \\
\tau_{12,v} & \tau_{2,v}   \\
\end{array}\right]$ dans le domaine de Siegel $F_{2}$ (voir paragraphe \ref{domaine archim}).  Dans cette uniformisation les produits de courbes elliptiques correspondent exactement au lieu $(\tau_{12}=0)$ dans l'ensemble $F_2$. On appelle alors \textit{simplicit\'e archim\'edienne} le produit :
\begin{equation}\label{simplicit\'e archim\'edienne}
 \sinf(A)=\prod_{v\in{M_{k}^{\infty}}}\vert\tau_{12,v}\vert^{d_{v}}.
\end{equation}

Il est donc ais\'e de voir que $\sinf(A)=0$ si et seulement si $A$ est un produit de courbes elliptiques. On appelle de plus \textit{trace archim\'edienne} de $A$ la quantit\'e :
\begin{equation}\label{trace archim\'edienne}
 \Trinf(A)=\sum_{v\in{M_{k}^{\infty}}}d_{v}\Tr(\Ima\tau_{v}).
\end{equation}
On note $D=2^8\disc(F)$ le discriminant de norme minimale d'un mod\`ele hyperelliptique entier $y^2=F(x)$ de la courbe sous-jacente. Notons de plus que les calculs explicites des hauteurs locales aux places finies sont bas\'es sur l'\'etude pouss\'ee de la surface de Kummer effectu\'ee par V. Flynn, N. Smart et M. Stoll dans les articles \cite{Fly, FlySma, Stoll1, Stoll2}. Dans le cas des jacobiennes de dimension 2 simples, le th\'eor\`eme prend la forme suivante :

\begin{theorem}\label{minoration dimension 2} (Version A.)
Soit $k$ un corps de nombres de degr\'e $d$. Soient $C/k$ une courbe de genre 2 admettant un point de Weierstrass rationnel sur $k$ et $A$ sa jacobienne. Alors si $A$ est g\'eom\'etriquement simple, il existe une constante $c_{1}(d)>0$ telle que pour tout point $P\in{A(k)}$ l'une des deux propositions suivantes est vraie :

\begin{tabular}{l}
\\
$(i)\;\;[n]P=O \;\textrm{pour un entier}\;\;1\leq n \leq 2\!\cdot\!10087^{4\cdot3^{16}d},$\\
\\

$(ii)\;\;\displaystyle{\widehat{h}_{A,2\Theta}(P)\geq c_{1}(d)\,\Big(\Trinf(A)-\frac{5}{3}\log \frac{\Nk(D)}{s_\infty(A)}\Big)}\;,$\\
\\
\end{tabular}

o\`u on peut prendre $c_{1}(d)=0,03/\left(d\,10087^{8\cdot3^{16}d}\right)$.

\vspace{0.3cm}

(Version B.)
Soit $k$ un corps de nombres de degr\'e $d$. Soient $C/k$ une courbe de genre 2 admettant un point de Weierstrass rationnel sur $k$ et $A$ sa jacobienne. Alors si $A$ est g\'eom\'etriquement simple, il existe une constante $c_2=c_{2}(d,A)>0$ telle que pour tout point $P\in{A(k)}$ l'une des deux propositions suivantes est vraie :

\begin{tabular}{l}
\\
$(i)\;\;\displaystyle{[n]P=O \;\textrm{pour un entier}\;\;1\leq n \leq 2\!\cdot\!10087^{4\cdot3^{16}d}\Big(\frac{\Nk(D)}{s_\infty(A)}\Big)^{10/3},}$\\
\\

$(ii)\;\;\displaystyle{\widehat{h}_{A,2\Theta}(P)\geq c_{2}\,\Trinf(A)}\;,$\\
\\
\end{tabular}

o\`u on peut prendre $c_{2}=0,03/\left(d\,10087^{8\cdot3^{16}d}\Nk(D)^{20/3}s_\infty(A)^{-20/3}\right)$.
\end{theorem}

\begin{Remarque}\label{remarque dim2} La version B du th\'eor\`eme fournit donc inconditionnellement une minoration non triviale de la hauteur des points rationnels sur une surface ab\'elienne simple. Un mod\`ele hyperelliptique entier d'une courbe de genre 2 ne v\'erifie pas n\'ecessairement l'in\'egalit\'e $\Trinf(A)>\frac{5}{3}\log \frac{\Nk(D)}{s_\infty(A)}$. On peut trouver des exemples parmi les courbes CM (\textit{i.e.} admettant des multiplications complexes), quitte \`a prendre une extension de corps. On sait par densit\'e des points CM dans $F_2$ qu'il en existe une infinit\'e telle que les jacobiennes associ\'ees $A$ v\'erifient $\Trinf(A)+\frac{5}{3}\log(s_\infty(A))>0$. On sait de plus que les vari\'et\'es ab\'eliennes CM ont potentiellement bonne r\'eduction partout. On montre alors dans le courant de la preuve du corollaire \ref{bonne r\'eduction potentielle} (qui se trouve juste apr\`es l'\'enonc\'e du corollaire \ref{langdim2}) qu'apr\`es une extension du corps $k$ de degr\'e uniform\'ement born\'e, on peut choisir un mod\`ele hyperelliptique avec discriminant minimal global trivial.
\end{Remarque}

\begin{Exemple}
Prenons par exemple $A=\Jac(C)$ o\`u $C$ est la courbe donn\'ee par le mod\`ele affine $y^2=x^5+x$. On sait calculer la matrice de p\'eriodes en MAGMA, qui fournit en valeur approch\'ee $\Trinf(A)\simeq1,88$ et $\log\vert s_\infty(A)\vert\simeq-0,75$, donc $\Trinf(A)+\frac{5}{3}\log(s_\infty(A))\simeq0,63>0$. De plus, $C$ est une courbe CM (on regarde le morphisme $(x,y)\rightarrow(\zeta^2 x,\zeta y)$ o\`u $\zeta$ est une racine primitive huiti\`eme de l'unit\'e). Sa jacobienne h\'erite donc de la structure CM et est en particulier potentiellement \`a bonne r\'eduction partout. 
\end{Exemple}

\begin{Remarque} L'existence d'un point de Weierstrass rationnel sur $k$ est \'equivalente \`a l'existence d'un mod\`ele $y^{2}=F(x)$ avec $\deg(F)=5$ sur $k$ plus une propri\'et\'e de sym\'etrie du diviseur $\Theta$.
\end{Remarque}

On d\'eduit imm\'ediatement de ce th\'eor\`eme le corollaire suivant :

\begin{corol}\label{borne torsion}
Soit $k$ un corps de nombres de degr\'e $d$. Soient $C/k$ une courbe de genre $2$ de mod\`ele entier
$y^{2}=F(x)$ avec $\deg(F)=5$ et $A/k$ sa jacobienne, g\'eom\'etriquement simple. Soient $\Trinf(A)$ sa trace archim\'edienne, $s_\infty(A)$ sa simplicit\'e archim\'edienne et $D=2^{8}\disc (F)$. On suppose que :
$$\Trinf(A)> \frac{5}{3}\log \frac{\Nk(D)}{s_\infty(A)}.$$
Alors on a :
$$\Card\Big(A(k)_{\mathrm{tors}}\Big)\leq 2^{4}\cdot10087^{16\cdot3^{16}d}.$$
\end{corol}

On compl\`ete le th\'eor\`eme \ref{minoration dimension 2} par l'\'etude de la situation du produit de courbes elliptiques, qui donne un th\'eor\`eme plus faible que celui de M. Hindry et J. Silverman dans \cite{HiSi3}, mais qui permet d'aboutir \`a un \'enonc\'e faisant intervenir les m\^emes quantit\'es que pour les jacobiennes simples. Introduisons de plus la quantit\'e $\hFprime(A/k)$, la \emph{hauteur de Faltings modifi\'ee} d'une vari\'et\'e ab\'elienne principalement polaris\'ee :
\begin{equation}\label{hauteur Faltings modifi\'ee}
\hFprime(A/k)=\hF(A/k)+\frac{1}{2d}\sum_{v\in{M_{k}^{\infty}}}d_{v}\log[\det(\Ima{\tau_{v}})].
\end{equation}
On donne alors la preuve du th\'eor\`eme de majoration suivant, bas\'e sur l'expression de la hauteur de Faltings donn\'ee dans \cite{Ueno} :

\begin{theorem} \label{faltings maj}
Soit $k$ un corps de nombres de degr\'e $d$. Soit $C/k$ une courbe de genre 2 avec bonne r\'eduction en $2$, prise dans un mod\`ele hyperelliptique entier $y^{2}=F(x)$ avec $\deg(F)=5$. On note $D=2^{8}\disc(F)$. On suppose que la jacobienne $A=\Jac(C)$ est g\'eom\'etriquement simple. Alors il existe des constantes $c_{3}(d)>0$ et $c_{4}(d)>0$ telles que :
$$\hFprime(A/k)\leq c_{3}\Trinf(A) + c_{4}\log \frac{\Nk(D)}{s_\infty(A)},$$

et on peut prendre : $c_{3}=\frac{6\pi}{10d}$ et $c_{4}=\frac{1}{10d}$.
\end{theorem}

Notons que ce th\'eor\`eme est un pas vers la conjecture 1.7 de S. David donn\'ee dans \cite{Dav2} page 513.
La conjonction des th\'eor\`emes \ref{minoration dimension 2} et \ref{faltings maj} fournit alors le corollaire suivant, dans lequel on fixe : si $A=\Jac(C)$ est la jacobienne d'une courbe de genre 2, avec $C$ donn\'ee
par un mod\`ele minimal entier $y^{2}=F(x)$ avec $\deg(F)=5$, on note $D=2^{8}\disc(F)$. Si $A=E_{1}\times E_{2}$
est un produit de courbes elliptiques, on note $D=\Delta_{E_{1}}\Delta_{E_{2}}$ le produit des discriminants minimaux de $E_{1}$ et $E_{2}$.

\begin{corol}\label{genre 2}

Soit $k$ un corps de nombres de degr\'e $d$ et $\varepsilon>0$. Soit $(A,\Theta)/k$ une vari\'et\'e ab\'elienne principalement polaris\'ee
de dimension 2. Si $A$ est simple, on suppose que $\Trinf(A)\geq (5/3 +\varepsilon) \log(\Nk(D)/s_\infty(A))$. Sinon on suppose $\Trinf(A)\geq(5/36+\varepsilon)\log\Nk(D)$. Alors il existe une constante $c(d,\varepsilon)>0$ telle que pour tout point $P\in{A(k)}$ v\'erifiant $\overline{\mathbb{Z}\!\cdot\! P}=A$ on a :
\[
\widehat{h}_{A,2\Theta}(P)\geq c \, \hFprime(A/k),
\]

et on peut prendre $c=0,015\frac{\varepsilon}{2+\varepsilon}\cdot10087^{-8\cdot3^{16}d}$.

\end{corol}

Les th\'eor\`emes 1 et 2, ainsi que ce corollaire, permettent de v\'erifier la conjecture de Lang et Silverman pour des familles infinies de vari\'et\'es ab\'eliennes de dimension 2, par exemple les jacobiennes simples, de simplicit\'e minor\'ee, qui ont potentiellement bonne r\'eduction partout :

\begin{corol}\label{bonne r\'eduction potentielle}
Soit $k$ un corps de nombres de degr\'e $d$. Soit $C/k$ une courbe de genre 2 donn\'ee dans un mod\`ele hyperelliptique entier $y^{2}=F(x)$ avec $\deg(F)=5$ et telle que $C/k$ a potentiellement bonne r\'eduction partout. Soit $A$ la jacobienne de $C$, g\'eom\'etriquement simple, de simplicit\'e archim\'edienne sup\'erieure \`a $1$. Alors il existe une constante $c(d)>0$ telle que pour tout point $P\in{A(k)}$ d'ordre infini on a :
\[
\widehat{h}_{A,2\Theta}(P)\geq c\,\hstprime(A),
\]
et on peut prendre $c=\frac{1}{20\pi}\cdot10087^{-320\cdot15^{16}d}$.

\end{corol}

Cet \'enonc\'e n'est pas couvert par le th\'eor\`eme de S. David \cite{Dav2} en dimension 2. Par contre le th\'eor\`eme de S. David se passe de l'hypoth\`ese archim\'edienne dont on a besoin pour mener \`a bien la strat\'egie locale.

On rajoute un dernier \'enonc\'e concernant les points rationnels sur les courbes de genre 2 :

\begin{corol}\label{points rationnels}
Soit $k$ un corps de nombres et $\varepsilon>0$. Soit $C/k$ une courbe de genre 2, avec bonne r\'eduction en toute place divisant $2$, donn\'ee dans un mod\`ele entier $y^{2}=F(x)$ avec $\deg(F)=5$. On notera $D=2^{8}\disc(F)$. Soit $A/k$ la jacobienne de $C$. On suppose que $\Trinf(A)\geq (5/3+\varepsilon)\log(\Nk(D)/\sinf(A))$. Alors il existe une constante $c_{2}(d,\varepsilon)>0$ ne d\'ependant que de $d=[k:\mathbb{Q}]$ et $\varepsilon$ telle que :
$$\Card(C(k))\leq c_{2}^{\rang A(k)+1},$$
et on peut choisir :
$$c_{2}=10087^{\displaystyle{(d+1)2^{35}}}\Big(1+\frac{2}{\varepsilon}\Big).$$
\end{corol}

\subsection{Domaine archim\'edien}\label{domaine archim}

Nous allons mettre en place une strat{\'e}gie de minoration proche de celle adopt{\'e}e par M. Hindry et J. Silverman dans le cas $g=1$ en utilisant la d\'ecomposition de la hauteur de N\'eron-Tate en hauteurs locales. Ce qui rend cette d{\'e}marche possible en dimension 1 est l'existence de formules explicites et relativement manipulables pour ces hauteurs locales. Bien qu'on ne dispose pas de formule dans le cas g{\'e}n{\'e}ral, on peut encore obtenir un \'enonc\'e en dimension 2.

Commençons par fixer le domaine de Siegel \label{domaine de Siegel} : soit $v$ une place archim\'edienne du corps $k$. On notera $H_{g}$ l'espace de Siegel associ\'e aux vari\'et\'es
ab\'eliennes sur $\bar{k}_{v}$ principalement polaris\'ees de dimension $g$ et munies d'une base symplectique (on pourra consulter \cite{LaBi} page 213). C'est l'ensemble des matrices $\tau=\tau_{v}$ de taille $g\times g$ sym\'etriques \`a coefficients complexes et v\'erifiant la condition $\Ima \tau >0$ (i.e. d\'efinies positives). Cet espace est muni d'une action transitive du groupe symplectique $\Gamma = \Sp(2g,\mathbb{R})$ donn\'ee par :
$$\left[ \begin{array}{cc}
A & B   \\
C & D   \\
\end{array}\right]\!\cdot\!\tau=(A\tau+B)(C\tau+D)^{-1}.$$

On consid\`ere alors $F_{g}$ un domaine fondamental pour l'action du sous-groupe $\Sp(2g,\mathbb{Z})$. On peut choisir $F_{g}$ de telle sorte qu'une matrice $\tau$ de ce domaine v\'erifie en particulier les conditions suivantes (voir \cite{Frei} page 34) :

\begin{itemize}
\item[$\bullet$] {S1 : Pour tout $\sigma\in{\Sp_{2g}(\mathbb{Z})}$ on a : $\det(\Ima(\sigma . \tau))\leq \det(\Ima(\tau))$. On dira que $\Ima\tau$ est \emph{maximale} pour l'action de $\Sp_{2g}(\mathbb{Z})$.}

\item[$\bullet$] {S2 : Si $\Ree(\tau)=(a_{i,j})$ alors $|a_{i,j}|\leq \frac{1}{2}$.}

\item[$\bullet$] {S3 : Si $\Ima(\tau)=(b_{i,j})$ alors pour tout $l\in\{1,...,g\}$ et tout $\zeta=(\zeta_{1},...,\zeta_{g})\in{\mathbb{Z}}^{g}$ tel que $\pgcd(\zeta_{1},...,\zeta_{l})=1$ on a $^{t}\zeta \Ima(\tau)\zeta \geq b_{l,l}$. De plus pour tout $i\in\{1,...,g\}$ on a $b_{i,i+1}\geq 0$. On a enfin $b_{g,g}\geq...\geq b_{1,1}\geq \sqrt{3}/2$ et $ b_{i,i}/2\geq|b_{i,j}|$.}
\end{itemize}

En dimension $g=2$ on aura en particulier les in\'egalit\'es, utilis\'ees constamment dans le texte (on note $\tau_{1}=\tau_{11}$ et $\tau_{2}=\tau_{22}$) :

\begin{equation}
\left\{ 
\begin{array}{l}
\Ima\tau_{2}\geq \Ima\tau_{1}\geq 2\Ima\tau_{12}\geq0, \\ 
\Ima\tau_{1}\geq \frac{\sqrt{3}}{2}. \\ 
\end{array}
\right.
\end{equation}

Dans tout le texte, les matrices $\tau$ seront toujours suppos\'ees appartenir au domaine fondamental $F_{2}$. On imposera de plus qu'elles soient de trace maximale. 

Dans la partie 2, suivant directement cette introduction, on d{\'e}compose la hauteur de N{\'e}ron-Tate
en hauteurs locales explicites. On s'inspire pour cela l'article de
E. V. Flynn et N. P. Smart \cite{FlySma}. On minore alors les hauteurs locales aux places finies en utilisant les r\'esultats de M. Stoll de l'article \cite{Stoll1}. La troisi\`eme
partie donne une autre d\'efinition de hauteur locale aux places archim\'ediennes. On r\'eunit les deux normalisations dans la quatri\`eme. Apr\`es avoir effectu\'e ces
minorations place par place,  on r{\'e}unit ces informations dans une
cinqui\`eme partie pour obtenir une minoration globale. On propose dans la sixi\`eme
partie une majoration de la hauteur de Faltings de la
jacobienne d'une courbe de genre 2. La septi\`eme partie regroupe des travaux parall\`eles sur les produits de courbes elliptiques. Enfin, on r\'eunit les r\'esultats
des parties 5, 6 et 7 dans une huiti\`eme partie regroupant trois corollaires : une minoration de la hauteur de N\'eron-Tate par la hauteur de Faltings, une borne sur la torsion des vari\'et\'es ab\'eliennes de dimension 2 et une borne sur le nombre de points rationnels d'une courbe de genre 2.

Terminons cette introduction en redonnant bri\`evement l'argument permettant de d\'eduire la structure des vari\'et\'es ab\'eliennes de dimension 2 principalement polaris\'ees. Soit $(A,\Theta)$ une telle vari\'et\'e. On a $\dime(\Theta)=1$ et $(\Theta)^{(2)}=2!=2$ (par Riemann-Roch, ou bien la formule de Poincar\'e \cite{LaBi} page 328). Si $\Theta$ est une courbe $C$ et $j:C\hookrightarrow \Jac(C)$ le plongement dans la jacobienne, on montre que $j(C)+j(C)$ est birationnellement \'equivalent \`a $A$, ce qui n'est possible que si $C$ est de genre 2 et $A\simeq \Jac(C)$. Si $\Theta=\sum C_{i}$ avec $C_{i}$ des courbes, on a :
$$2=(\Theta)^{(2)}=\sum (C_{i}\!\cdot\! C_{j}),$$
et chaque terme de la somme est un entier naturel. On d\'eduit alors que $\Theta$ est isomorphe \`a la somme de deux courbes, qui de plus sont des translat\'ees de sous-vari\'et\'es ab\'eliennes de A.

\vspace{0.4cm}
\textbf{Remerciements.} Merci \`a M. Hindry, G. R\'emond et J. Silverman pour leurs encouragements et leur int\'er\^et pour ce travail. Merci \`a l'arbitre anonyme de la publication pour ses remarques.

\vspace{0.2cm}

\section{Les hauteurs locales en dimension 2}

L'existence de la d\'ecomposition en hauteurs locales fait l'objet du th\'eor\`eme suivant (voir par exemple \cite{HiSi} page 242) :

\begin{thm} \label{d\'ecomposition}(N\'eron)  Soit $A/k$ une vari\'et\'e ab\'elienne d\'efinie sur un corps de
nombres $k$. Soit $M_{k}$ l'ensemble des places de $k$. Pour tout diviseur $\mathcal{D}$ sur $A$ on note $A_{\mathcal{D}}=A\backslash\supp(\mathcal{D})$. Alors pour toute place
$v\in{M_{k}}$ il existe une fonction hauteur locale, unique \`a une fonction additive constante pr\`es :
$$\widehat{\lambda}_{\mathcal{D},v}:A_{\mathcal{D}}(k_{v})\longrightarrow \mathbb{R}, $$
appel\'ee hauteur locale canonique, d\'ependant du choix de $\mathcal{D}$ et
v\'erifiant les propri\'et\'es suivantes, avec $\gamma_{i,v}$ des constantes d\'ependant de $v$ :

\begin{enumerate}[(i)]
\item $ \widehat{\lambda}_{\mathcal{D}_{1}+\mathcal{D}_{2},v}= \widehat{\lambda}_{\mathcal{D}_{1},v}+
  \widehat{\lambda}_{\mathcal{D}_{2},v}+\gamma_{1,v} $.
\item Si $\mathcal{D}=\divi(f)$, alors $\widehat{\lambda}_{\mathcal{D},v}=v\circ f+\gamma_{2,v}$.
\item Si $\Phi:B\rightarrow A$ est un morphisme de vari\'et\'es ab\'eliennes
  alors on a la relation : $ \widehat{\lambda}_{\Phi^{*}\mathcal{D},v}=
  \widehat{\lambda}_{\mathcal{D},v}\circ \Phi+\gamma_{3,v}$.
\item Soit $Q \in{A(k)}$ et soit $t_{Q}: A \rightarrow A $ la translation
  par $Q$. Alors on a la relation : $\widehat{\lambda}_{t_{Q}^{*}\mathcal{D},v}=
  \widehat{\lambda}_{\mathcal{D},v}\circ t_{Q}+\gamma_{4,v}$.
\item Soit $\widehat{h}_{A,\mathcal{D}}$ la hauteur globale canonique de $A$ associ\'ee
  \`a $\mathcal{D}$. Il existe une constante $c$ telle que, pour tout $P\in{A_{\mathcal{D}}(k)}$ :
  $$ \widehat{h}_{A,\mathcal{D}}(P)=\sum_{v\in{M_{k}}}n_{v} \widehat{\lambda}_{\mathcal{D},v}(P) + c. $$
\item{Si $\mathcal{D}$ v\'erifie $[2]^{*}\mathcal{D}=4\mathcal{D}+\divi(f)$ pour $f$ une fonction rationnelle sur $A$ et si l'on fixe les constantes de telle sorte qu'on ait la relation $\widehat{\lambda}_{\mathcal{D},v}([2]P)=4\widehat{\lambda}_{\mathcal{D},v}(P)+v(f(P))$, alors :
$$\widehat{h}_{A,\mathcal{D}}(P)=\sum_{v\in{M_{k}}}n_{v} \widehat{\lambda}_{\mathcal{D},v}(P).$$
(Notons que $f$ est unique \`a multiplication par une constante $a\in{k^{*}}$ pr\`es.) }
\end{enumerate} 

\end{thm}

Les deux premiers paragraphes sont directement issus de l'article de E.V. Flynn et N. Smart \cite{FlySma}. On en donne ici une reformulation un peu plus g\'eom\'etrique en omettant la plupart des preuves. Remarquons que l'article original \cite{FlySma} est \'ecrit pour $k=\mathbb{Q}$, mais on peut tout utiliser, \textit{mutatis mutandis}, sur un corps de nombres $k$. Ceci est en fait d\'ecrit dans les articles de M. Stoll \cite{Stoll1} et \cite{Stoll2}. 

\subsection{Jacobienne et surface de Kummer}

On se donne une courbe $C$ de genre 2 sur un corps de nombres $k$. On sait que $C$ est hyperelliptique, elle poss\`ede donc six \emph{points de Weierstrass}, les points fixes de l'involution hyperelliptique. On fait l'hypoth\`ese que l'un de ces points, appel\'e $P_{0}$, est rationnel sur $k$. On note $\cl$ pour la classe rationnelle d'un diviseur. On d\'efinit alors le plongement jacobien de la courbe $C$ dans sa jacobienne :
$$j : C\hookrightarrow \Jac(C)$$
$$\hspace{0.6cm}P\mapsto \cl\Big((P)-(P_{0})\Big).$$

On d\'efinit alors $\Theta=j(C)$.

\begin{Remarque} Ce choix de $P_{0}$ permet d'affirmer que : $P\in{\Theta} \iff -P\in{\Theta}.$
\end{Remarque}

E.V. Flynn et N. Smart explicitent dans l'article \cite{FlySma} un choix possible des fonctions hauteurs locales lorsque
$A$ est la jacobienne d'une courbe
de genre 2. Nous suivrons pour cela leur normalisation 
pour les hauteurs locales. Le
diviseur qu'ils utilisent est $\mathcal{D}=2\Theta$ lorsque le mod\`ele hyperelliptique est de degr\'e 5. Soulignons que ce choix de diviseur est unique \`a translation par un point de $2$-torsion pr\`es.

Soient $k$ un corps de nombres et $C/k$ une courbe de genre 2. On peut identifier la jacobienne $\Jac(C)$ au carr\'e sym\'etrique de la courbe, $\Sym ^{2}(C)$, dans lequel il faut contracter un diviseur (qui correspond au diviseur exceptionnel d'un \'eclatement d'un point de $\Jac(C)$). Ce proc\'ed\'e est bien d\'ecrit dans \cite{Mum4} page 52. La surface de Kummer $K$ est d\'efinie comme le quotient $\Jac(C)/(\pm 1)$. Elle se plonge dans $\mathbb{P}^{3}$. Voyons cela plus en d\'etails : comme on a suppos\'e que $P_{0}$ est un point de Weierstrass rationnel sur $k$, on peut se donner un mod\`ele hyperelliptique de la courbe $C$ entier sur $k$ de degr\'e impair, avec $a_{5}\neq 0$ et sans racine multiple :
$$C : y^{2}=F(x)=a_{5}x^{5}+a_{4}x^{4}+a_{3}x^{3}+a_{2}x^{2}+a_{1}x+a_{0}.$$
Contrairement au mod\`ele plus g\'en\'eral de degr\'e 6, il n'y a dans ce mod\`ele qu'un point \`a l'infini : $P_{0}=\infty$. L'\'etude de \cite{FlySma} est men\'ee en degr\'e 6, le cas quintique est plus simple et inclus dans leur travail (il suffit de sp\'ecialiser $f_{6}=0$ dans leur notation).

On note $A=\Jac(C)$ la jacobienne de $C$. L'involution
hyperelliptique donn\'ee sur la courbe $C$ par $i:(x,y) \rightarrow (x,-y)$ induit la multiplication par $[-1]$ sur $A$. On consid\`ere le quotient de $A$ par $(\pm 1)$. La surface $K$ est donn\'ee par l'\'equation quartique homog\`ene suivante (donn\'ee dans \cite{Fly} ou \cite{CasFly} page 19 et reprise dans l'annexe de \cite{Paz}) :
$$R(k_{1},k_{2},k_{3})k_{4}^{2}+S(k_{1},k_{2},k_{3})k_{4}+T(k_{1},k_{2},k_{3})=0.$$

On peut donner les points de $K$ par l'application :
$$\kappa:\mathrm{Sym}^{2} (C)\longrightarrow K\subset \mathbb{P}^{3}$$ 
$$\kappa:P=(P_{1},P_{2})\longmapsto K_{P}=(k_{1},k_{2},k_{3},k_{4}),$$
o\`u on a d\'efini pour un point $P=((x_{1},y_{1}),(x_{2},y_{2}))$ hors du support du diviseur $\Theta$ :

\begin{center} \;$\left\{ 
\begin{tabular}{l}
$k_{1}=1,$ \\ 
$k_{2}=x_{1}+x_{2},$ \\ 
$k_{3}=x_{1}x_{2},$\\
$k_{4}=\left(\begin{tabular}{l}
$2a_{0}+a_{1}(x_{1}+x_{2})+2a_{2}x_{1}x_{2}+a_{3}(x_{1}^{2}x_{2}+x_{1}x_{2}^{2})$\\
$+ 2a_{4}x_{1}^{2}x_{2}^{2}+a_{5}(x_{1}^{3}x_{2}^{2}+x_{1}^{2}x_{2}^{3})-2y_{1}y_{2}$
\end{tabular}\right)/(x_{1}-x_{2})^{2}$\\
\end{tabular}
\right.$
\end{center}

Pour un point $P=((x_{1},y_{1}),\infty)$ : $k_{1}=0,\quad k_{2}=1, \quad k_{3}=x_{1}, \quad k_{4}=a_{5}x_{1}^{2}.$
Le diviseur $\mathcal{D}'$ sur $\Sym^{2}(C)$ associ\'e \`a $(k_{1}=0)$ est donn\'e par $\mathcal{D}'\!=\!2(C\times \{\infty\})$. Ce diviseur $\mathcal{D}'$ s'envoie donc \textit{via} l'application $\pi:\Sym^{2}(C)\rightarrow \Jac(C)$ sur le diviseur $\mathcal{D}=2\Theta$.

Choisissons alors un plongement $k\hookrightarrow_{v} \mathbb{C}$. Les points complexes de $(\Jac(C),\Theta)$ forment un tore complexe qu'on normalise ainsi : $\Jac(C)(\mathbb{C})\simeq A_{\tau_{v}}(\mathbb{C})=\mathbb{C}^{2}/\mathbb{Z}^{2}+\tau_{v}\mathbb{Z}^{2}$, avec $\tau_{v}$ une matrice obtenue en calculant les p\'eriodes de la surface de Riemann compacte $C(\mathbb{C})$. Le diviseur $\Theta (\mathbb{C})$ est alors identifi\'e \`a la courbe $C(\mathbb{C})\hookrightarrow A_{\tau_{v}}(\mathbb{C})$.

\subsection{Hauteurs}\label{normalisation}

On garde le cadre pr\'ec\'edent et on d\'efinit suivant \cite{FlySma} les hauteurs naïve et canonique d'un point $P=(P_{1},P_{2})\in{A(k)}$. On va normaliser le point projectif $K_{P}$ en fixant la premi\`ere coordonn\'ee non nulle comme \'etant \'egale \`a 1 (c'est la normalisation choisie dans \cite{FlySma}). On peut donc d\'efinir la hauteur naïve comme \'etant :
$$h_{K}(P)=h(K_{P})=\sum_{v\in{M_{k}}}n_{v}\log \max_{i\in\{1,..,4\}}(|k_{i}|_{v}),$$
et la hauteur canonique associ\'ee :
$$\widehat{h}_{A,2\Theta}(P)=\lim_{n\rightarrow +\infty}\frac{h(K_{[2^{n}]P})}{4^{n}},$$
o\`u on note $K_{[2^{n}]P}$ l'image sur la surface de Kummer de la multiplication par $[2^{n}]$ d'un point $P$ de la jacobienne ; la surface de Kummer n'a plus la structure de groupe de la jacobienne, mais on peut passer l'application au quotient :

\dgARROWLENGTH=0.5cm 

\[
\begin{diagram}
\node{A}
\arrow[3]{e,t}{[2^{n}]}
\arrow[3]{s}
\node[3]{A}
\arrow[3]{s}
\\\\\\
\node{A/(\pm 1)}
\arrow[3]{e}
\node[3]{A/(\pm 1)}
\end{diagram}
\]

On a choisi de travailler avec la multiplication par $[2]$. En effet il existe des formules explicites de duplication sur la surface de Kummer : prenons un point $K_{P}$, alors la formule de duplication est donn\'ee par des polynômes homog\`enes explicites (donn\'es sur le site internet de V. Flynn et reproduites en annexe de \cite{Paz}) not\'es $\delta(K_{P})=(\delta_{1},\delta_{2},\delta_{3},\delta_{4})$ de degr\'e total 4 en les $k_{i}$. Avec la normalisation choisie ici, on aura donc (lorsque $P$ et $[2]P$ sont hors du support du diviseur $\Theta$): 
$$K_{[2]P}=\frac{\delta(K_{P})}{\delta_{1}(K_{P})}.$$

La hauteur locale naïve en une place $v$ est d\'efinie par :
$$\lambda_{2\Theta,v}:(k_{1},k_{2},k_{3},k_{4})\longmapsto \log \max_{i\in{\{1,..,4\}}}(|k_{i}|_{v}).$$

\begin{Remarque} Cette construction doit \^etre vue comme l'analogue de la hauteur locale sur une courbe elliptique $\lambda_{v}(P)=\log|x(P)|_{v}$, o\`u $x(P)$ est la coordonn\'ee d'un point $P$ dans un mod\`ele de Weierstrass.
\end{Remarque}

Calculons alors :

\begin{tabular}{lll}
$\displaystyle{\lambda_{2\Theta,v}(K_{[2]P})-4\lambda_{2\Theta,v}(K_{P})}$ & $=$ & $\displaystyle{\log\frac{\max_{i\in{\{1,..,4\}}}(|k_{i}([2]P)|_{v})}{\max_{i\in{\{1,..,4\}}}(|k_{i}(P)|^{4}_{v})}}$\\
\\

& $=$ & $\displaystyle{-\log|\delta_{1}(K_{P})|_{v}+\log\frac{\max_{i\in{\{1,..,4\}}}(|\delta_{i}(K_{P})|_{v})}{\max_{i\in{\{1,..,4\}}}(|k_{i}(P)|^{4}_{v})}.}$\\
\\

\end{tabular}

Toujours en suivant \cite{FlySma} on d\'efinit la hauteur locale canonique d'un point $P\in{A_{2\Theta}(k)}$ comme suit, en posant tout d'abord :
$$E_{v}(K_{P}):=\frac{\max(|\delta_{i}(K_{P})|_{v})}{\max(|k_{i}(P)|_{v}^{4})}$$
et :
$$\mu_{2\Theta,v}(K_{P}):=\sum_{n=0}^{+\infty}\frac{1}{4^{n+1}}\log\Big(E_{v}(K_{[2^{n}]P})\Big).$$
\begin{Remarque} Cette quantit\'e $\mu_{2\Theta,v}(K_{P})$ ne d\'epend pas de la normalisation du point projectif. Il est int\'eressant de remarquer que la preuve du lemme 3 de \cite{FlySma} utilise une normalisation diff\'erente du reste de l'article.
\end{Remarque}

Alors on d\'efinit la hauteur locale canonique :
$$\widehat{\lambda}_{2\Theta,v}(P)=\lambda_{2\Theta,v}(K_{P})+\mu_{2\Theta,v}(K_{P}).$$
Lorsque $K_{[2]P}$ est lui aussi hors du support du diviseur $2\Theta$, cette hauteur locale canonique v\'erifie l'\'equation fonctionnelle :
$$\widehat{\lambda}_{2\Theta,v}([2]P)-4\widehat{\lambda}_{2\Theta,v}(P)=-\log|\delta_{1}(K_{P})|_{v}=v(f(P)),$$
avec $f(P):=\delta_{1}(K_{P})$ et $\divi(f)=[2]^{*}(2\Theta)-4(2\Theta)$.

D'apr\`es le th\'eor\`eme 4 de l'article \cite{FlySma} (qui ne d\'epend pas de la normalisation projective choisie pour le point $P$) on a bien pour $P\in{A_{\Theta}(k)}$ (i.e. hors du support du diviseur $\Theta$) :
$$\widehat{h}_{A,2\Theta}(P)=\sum_{v\in{M_{k}}}n_{v}\widehat{\lambda}_{2\Theta,v}(P).$$

\section{Une autre hauteur locale archim\'edienne}

On donne dans cette partie une autre normalisation des hauteurs locales archim\'ediennes, grâce \`a l'utilisation des fonctions th\^eta. Ce lien est donn\'e par A. N\'eron, voir par exemple l'article fondateur \cite{Neron} page 329.

\subsection{D\'efinition}\label{d\'efinition fonctions th\^eta}

On commence ce paragraphe par rappeler la d\'efinition des fonctions th\^eta : soient $Z\in{\mathbb{C}^{2}}$ et $\tau\in{F_{2}}$ :
$$\theta_{a,b}(Z)=\sum_{n\in{\mathbb{Z}^{2}}}e^{2i\pi \left(\frac{1}{2}\, ^{t}(n+a)\tau(n+a)+\, ^{t}(n+a)(Z+b)\right)},$$ o\`u $a,b\in{\frac{1}{2}\mathbb{Z}^{2}/\mathbb{Z}^{2}}$ forment le \emph{vecteur caract\'eristique} de la fonction th\^eta.

Tout vecteur complexe $Z$ peut se d\'ecomposer en $Z=X+\tau Y$ avec $X,Y\in{\mathbb{R}^{2}}$. 

Le th\'eor\`eme de Riemann (voir par exemple \cite{LaBi} page 330) montre que les points complexes du diviseur $\Theta$ sont les z\'eros d'une fonction th\^eta avec caract\'eristique (la caract\'eristique fixant le point de torsion par lequel il faut \'eventuellement translater, voir par exemple \cite{Mum4} page 60 et page 69 et \cite{Mum1} page 164).

En se reportant \`a l'analyse men\'ee dans \cite{Mum1} page 164 et \cite{Mum2} page 3.80-82, on peut identifier le vecteur caract\'eristique comme \'etant $[a,b]=[1/2, \; 1/2 , \; 1, \; 1/2]$. C'est aussi le choix qui est fait dans \cite{Yosh}. Il est de plus \'equivalent de prendre la troisi\`eme coordonn\'ee \'egale \`a z\'ero.

Fixons alors $[a,b]=[1/2, \; 1/2, \; 0, \; 1/2].$ Cette
caract\'eristique est impaire, la fonction th\^eta consid\'er\'ee
v\'erifie en particulier $\theta_{a,b}(0)=0$. Son diviseur est $\Theta(\mathbb{C})$ et il contient $O$ dans son support.

Soit $k$ un corps de nombres. Soient $C/k$ une courbe de genre 2 et $A=\Jac(C)$ sa jacobienne, polaris\'ee par $\Theta$. Soit $v$ une place archim\'edienne et soit $\tau_{v}$ l'\'el\'ement de $F_{2}$ correspondant \`a $(A(\bar{k_{v}}),\Theta)$. On peut alors donner la d\'efinition suivante :

\begin{propdef}\label{autre hauteur archim} \`A une constante pr\`es, la hauteur locale associ\'ee au diviseur $\Theta$ pour la place $v\in{M_{k}}$ archim\'edienne peut s'exprimer comme suit, pour tout point $P$ hors du support du diviseur $\Theta$ et toute coordonn\'ee complexe de $P$ not\'ee $Z(P)$ :
$$\Lambda_{\Theta,v}(P)=-\log\left(\Big|\theta_{a,b}(Z(P))\Big|_{v}\;e^{-\pi \,^{t}\Ima Z(\Ima\tau_{v})^{-1}\Ima Z}\right).$$
\end{propdef}

On peut trouver cette id\'ee d'\'ecriture de la hauteur locale dans l'article \cite{Neron}, page 329. Cette fonction est bien une fonction sur le tore, on a corrig\'e la fonction th\^eta de telle sorte qu'elle soit $\mathbb{Z}^{2}+\tau_{v}\mathbb{Z}^{2}$-p\'eriodique. Elle v\'erifie de plus l'\'equation fonctionnelle :
$$\Lambda_{\Theta,v}([2]P)-4\Lambda_{\Theta,v}(P)=-\log\frac{|\theta_{a,b}(2Z(P))|_{v}}{|\theta_{a,b}(Z(P))|^{4}_{v}}=v(f(P)),$$
avec $f(P):=\theta_{a,b}([2]P)/\theta_{a,b}(P)^4$ et $\divi(f)=[2]^{*}\Theta-4\Theta$.

\section{Diff\'erences de hauteurs locales}

On montre dans cette partie comment tirer parti \`a la fois des informations aux places finies issues de la normalisation des hauteurs locales au sens de Flynn-Smart (donn\'ee dans le paragraphe \ref{normalisation}) et des calculs men\'es sur les fonctions th\^eta.

\subsection{Discussion autour de la torsion}

Rappelons la notation $A_{\mathcal{D}}(k)=A(k)\backslash \mathcal{D}(k)$.

\begin{prop}\label{\'egalit\'e locale}

Soit $v$ une place archim\'edienne. Soit $\widehat{\lambda}_{2\Theta,v}$ la hauteur locale canonique normalis\'ee au sens de Flynn-Smart et d\'efinie dans la partie \ref{normalisation}. Soit $\Lambda_{\Theta,v}$ la hauteur locale archim\'edienne d\'efinie en \ref{autre hauteur archim}. Il existe une constante $C_{\infty, v}$ telle que :
\begin{equation}\label{Cinf}
\forall P\in{A_{\Theta}(k)}, \;\; \widehat{\lambda}_{2\Theta,v}(P)=2\Lambda_{\Theta,v}(P)+C_{\infty,v}.
\end{equation}
\end{prop}

\begin{proof}
C'est en fait un simple corollaire du th\'eor\`eme \ref{d\'ecomposition}.
\end{proof}

Pour obtenir une minoration de la hauteur locale archim\'edienne normalis\'ee comme dans la partie \ref{normalisation}, il suffira donc de minorer la hauteur locale archim\'edienne $\Lambda_{\Theta,v}$ et la constante $C_{\infty,v}$. Nous allons estimer cette constante en particularisant l'\'equation donn\'ee dans la proposition \ref{\'egalit\'e locale} en des points de torsion. Il faut cependant s'assurer que les points ne sont pas sur le support du diviseur $\Theta$.

Nous allons utiliser le fait suivant :

\begin{prop}(Boxall, Grant)\label{3-torsion}
Soit $\Jac(C)/k$ une jacobienne de dimension 2 sur un corps quelconque, simple et polaris\'ee par le diviseur $\,\Theta=C$. Alors aucun point d'ordre 3 n'est sur le diviseur $\Theta$.
\end{prop}

\begin{proof} Il suffit de consulter la preuve de la proposition 1.5 de \cite{BoxGra}. Une deuxi\`eme preuve de cette proposition figure en corollaire du lemme de z\'eros \ref{point hors de la courbe} du pr\'esent texte.
\end{proof}

\begin{Remarque} La situation est compl\`etement diff\'erente sur un produit de courbes elliptiques $E_{1}\times E_{2}$ polaris\'e par $E_{1}\times \{O\}+\{P_{1}\}\times E_{2}$, o\`u $P_{1}$ est un point de $2$-torsion non nul. En effet les points de la forme $(R,O)$, avec $3R=O$, sont des points de $3$-torsion qui sont sur le diviseur. Quitte \`a \'etendre un peu le corps, il y a donc 9 points de $3$-torsion sur les produits de courbes elliptiques ainsi polaris\'es.
\end{Remarque}

Revenons aux vari\'et\'es ab\'eliennes simples en dimension 2. Nous pouvons nous baser sur la derni\`ere proposition et utiliser les points de $3$-torsion dans l'\'etude de la constante de normalisation des hauteurs locales. En particularisant l'\'egalit\'e (\ref{Cinf}) pour $R$ un point de $3$-torsion non nul nous obtenons :
$$C_{\infty,v}=\widehat{\lambda}_{2\Theta,v}(R)-2\Lambda_{\Theta,v}(R),$$
ce qui implique donc que pour tout point $P\in{A_{\Theta}(k)}$ et tout point $R$ d'ordre 3 la diff\'erence est constante (on n'utilise que le fait que $R\notin{\Theta}$ pour l'instant) :
\begin{equation}\label{differences archim}
\widehat{\lambda}_{2\Theta,v}(P)-\widehat{\lambda}_{2\Theta,v}(R)=2\Lambda_{\Theta,v}(P)-2\Lambda_{\Theta,v}(R).
\end{equation}

Posons $\mathcal{T}_{3}$ l'ensemble des points d'ordre 3. C'est un ensemble de cardinal 80, car il y a 81 points de $3$-torsion mais le point $O$ n'est pas d'ordre exactement 3. Nous allons \`a pr\'esent effectuer le calcul clef de notre strat\'egie d'\'etude de la hauteur globale. Commençons par des lemmes concernant les points d'ordre 3.

\begin{Remarque}
Notons que nous allons supposer ici que $k(A[3])=k$. Nous verrons \`a la fin du paragraphe \ref{preuve du thm} (preuve du th\'eor\`eme \ref{minoration dimension 2}) comment nous passer de cette hypoth\`ese.
\end{Remarque}

\begin{lem}\label{3-torsion delta}
Soit $A/k$ une vari\'et\'e ab\'elienne de dimension 2, simple et principalement polaris\'ee. Soit $R\in{A(k)}$ un point de $3$-torsion non nul. Soit $\widehat{\lambda}_{2\Theta,v}$ la hauteur locale normalis\'ee comme dans la partie \ref{normalisation}. Alors on a :
$$\widehat{\lambda}_{2\Theta,v}(R)=\frac{1}{3}\log\Big|\delta_{1}(K_{R})\Big|_{v}.$$
\end{lem}

\begin{proof}
On note $v(f(P))=-\log|f(P)|_{v}$.
Il suffit de partir de l'\'equation fonctionnelle fixant la hauteur locale :
$$\widehat{\lambda}_{2\Theta,v}([2]R)-4\widehat{\lambda}_{2\Theta,v}(R)=v(f(R)).$$
Comme $R$ est un point de $3$-torsion non nul, on a $[2]R=-R$. De plus le diviseur $\Theta$ est sym\'etrique et d\'efini grâce \`a un point de Weierstrass donc la hauteur locale est paire, ce qui implique :
$$-3\widehat{\lambda}_{2\Theta,v}(R)=v(f(R)),$$
d'o\`u le r\'esultat, en notant que dans la normalisation \ref{normalisation}, $f(P)=\delta_{1}(K_{P})$.
\end{proof}

\begin{lem}\label{36}
Soient $k$ un corps de nombres et $C/k$ une courbe de genre 2, dont on se donne un mod\`ele hyperelliptique entier $y^{2}=F(x)=a_{5}x^{5}+...+a_{0}$. Soit $A$ la jacobienne de $C$. On note $D=2^{8}\disc(F)$. Alors on a l'\'egalit\'e :
$$\prod_{R\in{A[3]\backslash\{O\}}}\delta_{1}(K_{R})=\frac{1}{3^{24}} D^{36}.$$
\end{lem}

\begin{proof}
On sait que si $D\neq 0$, la courbe $C$ est lisse et les points d'ordre exactement 3 de $\Jac(C)$ ne sont pas sur le support du diviseur $\Theta$. Ceci implique que $-R=[2]R$ n'est pas sur le support de $\Theta$, donc $\delta_{1}(R)\neq 0$ pour tout point d'ordre 3.
En contraposant on obtient l'implication : $\Big(\delta_{1}(R)=0\Big)\Rightarrow \Big( D=0\Big)$.
On sait de plus que $\delta_{1}(K_{R})\in{\mathbb{Z}[k_{1},...,k_{4}][a_{0},...,a_{5}]}$ et le degr\'e total en les $a_{i}$ est 3. Le degr\'e total en les $k_{i}$ est 4. Soit $L=\mathbb{Q}(a_{0},...,a_{5})$. Les $k_{i}$ \'etant les coordonn\'ees des points d'ordre $3$ ce sont des \'el\'ements alg\'ebriques sur $L$. On sait de plus que $[L[A[3]]:L]\leq 3^{16}$. Posons : 
$$Q(a_{0},...,a_{5})=\prod_{R\in{A[3]\backslash\{O\}}}\delta_{1}(K_{R}). $$ 
L'ensemble $A[3]\backslash\{O\}$ est stable sous l'action du groupe de Galois $\Gal(\bar{L}/L)$. 
On peut en d\'eduire que $Q\in{\mathbb{Q}(a_{0},...,a_{5})}$. Or $Q$ est une fraction rationnelle sans pôle : c'est donc un polynôme. On en d\'eduit que $Q(a_{0},...,a_{5})\in{\mathbb{Q}[a_{0},...,a_{5}]}$. 

D'autre part on a $D\in{\mathbb{Z}[a_{0},...,a_{5}]}$ et $D$ est irr\'eductible. Ceci permet de dire qu'il existe des constantes universelles $c_{0}\in{\mathbb{Q}}$ et $d_{0}\in{\mathbb{N}}$ telles que :
$$Q=c_{0}D^{d_{0}}.$$
(Notons que seules les puissances de $D$ ne s'annulent sur aucun point d'ordre 3, voir l'article \cite{Grant2}. On peut aussi raisonner ainsi : il suffit de montrer l'\'egalit\'e sur $\mathbb{C}$, ce qui est faisable en \'etudiant le poids, les z\'eros et les p\^oles de la forme $Q/D^{36}$.)

Dans un deuxi\`eme temps on cherche \`a expliciter les constantes $c_{0}$ et $d_{0}$. On sait que $D$ est de poids 40 et $\delta_{1}$ de poids 18. Comme il y a 80 termes $\delta_{1}(R)$, cela montre que $d_{0}=36$. Il suffit ensuite de mener le calcul complet dans un cas particulier. Nous allons choisir l'\'equation $y^{2}=x^{5}-1$. On pose $F(x)=x^{5}-1$. On a donc $\disc(F)=3125=5^{5}$.

On va calculer les coordonn\'ees $K_{R}$ des points $R$ d'ordre 3 exactement pour cet exemple particulier. On note $K_{R}=(1,k_{2},k_{3},k_{4})$ la coordonn\'ee normalis\'ee. Ces coordonn\'ees v\'erifient tout d'abord l'\'equation de la surface de Kummer (donn\'ee dans \cite{CasFly} page 19 ou dans l'annexe de \cite{Paz}), qu'on notera $\delta_{0}=0$. Ces points v\'erifient de plus l'\'equation $K_{[2]R}=K_{R}.$
On doit donc r\'esoudre le syst\`eme suivant :

\begin{center}
 $\left\{
 \begin{tabular}{lll}

 $\delta_{2}-k_{2}\delta_{1}$ & $=$ & $0,$\\

 $\delta_{3}-k_{3}\delta_{1}$ & $=$ & $0,$\\

 $\delta_{4}-k_{4}\delta_{1}$ & $=$ & $0,$\\

 $\delta_{0}$ & $=$ & $0.$\\
 \end{tabular}
 \right.$
\end{center}

On utilise alors les formules de duplication sur la surface de Kummer donn\'ees en annexe de \cite{Paz}, dans lequelles on sp\'ecialise ainsi : $a_{6}=0$, $a_{5}=1$, $a_{4}=a_{3}=a_{2}=a_{1}=0$, $a_{0}=-1$ et $k_{1}=1$. A partir de l\`a, on s'est ramen\'e au probl\`eme de la recherche de racines communes \`a quatre polynômes fix\'es d\'ependant de trois variables $k_{2}$, $k_{3}$ et $k_{4}$.

On peut r\'esoudre ce syst\`eme en utilisant une technique de r\'esultants : on prend le r\'esultant des deux premiers polynômes par rapport \`a la premi\`ere variable, puis le r\'esultant du r\'esultat avec le troisi\`eme polynôme par rapport \`a la deuxi\`eme variable et un dernier r\'esultant en fonction de la derni\`ere variable. On fait cela dans tous les ordres possibles. Ceci donne des valeurs possibles pour la derni\`ere variable, on remonte ensuite les calculs et on v\'erifie \textit{a posteriori} que les coordonn\'ees candidates sont bien des solutions des quatre \'equations de d\'epart. Une fois les coordonn\'ees trouv\'ees, le calcul de $\delta_{1}$ est direct.

Les calculs ont \'et\'e men\'es compl\`etement en utilisant le logiciel PARI. Le r\'esultat est le suivant :
$$Q(1,0,0,0,0,-1)=2^{288}3^{-24}5^{180}.$$
Ceci fournit, puisqu'on a $D=2^{8}\disc(F)=2^{8}5^{5}$, les valeurs $c_{0}=3^{-24}$ et $d_{0}=36$.
\end{proof}

\begin{defin}
On note $\mathcal{Z}_{2}$ l'ensemble des 10 caract\'eristiques paires en dimension 2, et $\theta_{m}(0,\tau)$ la constante th\^eta associ\'ee \`a la caract\'eristique $m$ et la matrice de p\'eriodes $\tau$. On d\'efinit alors le discriminant modulaire comme \'etant :
$$\Delta(\tau)=2^{-12}\prod_{m\in{\mathcal{Z}_{2}}}\theta_{m}(0,\tau)^2.$$
\end{defin}

\begin{lem}\label{formule grant}
Soit $A/k$ une surface ab\'elienne simple sur un corps de nombres $k$. Soit $v$ une place archim\'edienne et soit $A(\bar{k}_{v})\simeq \mathbb{C}^2/\mathbb{Z}^2+\tau_{v}\mathbb{Z}^2$ une uniformisation complexe. Pour tout point $R$ d'ordre 3, on note $Z_{R}=X_{R}+\tau_{v}Y_{R}$ sa coordonn\'ee complexe, avec $X_{R}, Y_{R}\in{\mathbb{R}^2}$. Posons $[a,b]=[1/2,1/2,0,1/2]$. On a alors l'\'egalit\'e suivante : $$\prod_{R\in{T_{3}}}\Big|\theta_{a,b}(R,\tau_{v})\Big|_{v} e^{-\pi ^tY_{R}\Ima\tau_{v} Y_{R}}=3^4 |\Delta(\tau_{v})|_{v}^4.$$ 
\end{lem}

\begin{proof}
Il suffit d'utiliser le th\'eor\`eme 2 page 234 de l'article \cite{Grant1} et les transformations classiques des fonctions th\^eta.
\end{proof}

{\begin{clef} \label{clef} Soit $P\in{A(k)}$ et soit $n\geq 1$ tel que $[n]P\in{A_{\Theta}(k)}$ :
$$n^2\widehat{h}_{A,2\Theta}(P)=\!\!\sum_{v\in{M_{k}^{0}}}\!\!n_{v}\widehat{\lambda}_{2\Theta,v}([n]P)+\!\!\sum_{v\in{M_{k}^{\infty}}}\!\!n_{v}2\Lambda_{\Theta,v}([n]P)+\frac{3}{20d}\log\Nk(D)+\frac{1}{10}\!\!\sum_{v\in{M_{k}^{\infty}}}\!\!n_{v}\log|\Delta(\tau_{v})|_{v}.$$
\end{clef}}

\begin{proof}
On commence par \'ecrire $$n^2\widehat{h}_{A,2\Theta}(P)=\frac{1}{80}\sum_{R\in{\mathcal{T}_{3}}}\Big(\widehat{h}_{A,2\Theta}([n]P)-\widehat{h}_{A,2\Theta}(R)\Big)=\frac{1}{80}\sum_{R\in{\mathcal{T}_{3}}}\sum_{v\in{M_{k}}}n_{v}\Big(\widehat{\lambda}_{2\Theta,v}([n]P)-\widehat{\lambda}_{2\Theta,v}(R)\Big).$$
On utilise alors l'\'equation (\ref{differences archim}) aux places archim\'ediennes pour obtenir :

\flushleft
\begin{tabular}{ll}
$n^2\widehat{h}_{A,2\Theta}(P)=$ & $\displaystyle{\sum_{v\in{M_{k}}}n_{v}\widehat{\lambda}_{2\Theta,v}([n]P)-\frac{1}{80}\sum_{R\in{\mathcal{T}_{3}}}\sum_{v\in{M_{k}^{0}}}n_{v}\widehat{\lambda}_{2\Theta,v}(R)+\sum_{v\in{M_{k}}}2n_{v}\Lambda_{\Theta,v}([n]P)}$\\
& $\displaystyle{-\frac{1}{80}\sum_{R\in{\mathcal{T}_{3}}}\sum_{v\in{M_{k}^{\infty}}}2n_{v}\Lambda_{\Theta,v}(R)}$\\
\end{tabular}

\`A ce stade, il suffit d'appliquer les lemmes \ref{3-torsion delta} et \ref{36} pour la somme sur les points d'ordre 3 aux places finies et le lemme \ref{formule grant} pour la somme sur les points d'ordre 3 aux places archim\'ediennes. On remarquera que les puissances de 3 disparaissent dans les constantes.
\end{proof}

Nous allons diviser le travail de minoration de la hauteur globale en quatre tâches :

\begin{enumerate}

\item Pour $v$ une place finie : minorer $\widehat{\lambda}_{2\Theta,v}([n]P)$.

\item Pour $v$ une place archim\'edienne : minorer $\Lambda_{\Theta,v}([n]P)$.

\item Pour $v$ une place archim\'edienne : minorer $|\Delta(\tau_{v})|$.

\item Redescendre sur le corps de base.

\end{enumerate}

On traite du premier point dans le prochain paragraphe. Les \'etudes 2 et 3 aux places archim\'ediennes feront l'objet des parties suivantes. Le quatri\`eme point sera trait\'e \`a la fin de la preuve du th\'eor\`eme \ref{minoration dimension 2}.

\subsection{Estimation aux places finies}

\begin{prop}\label{minoration finie}

La hauteur locale en une place finie, normalis\'ee comme dans \ref{normalisation}, peut \^etre minor\'ee de la façon suivante (pour $P$ hors du support du diviseur $\Theta$) :
$$\widehat{\lambda}_{2\Theta,v}(P)\geq -\frac{1}{3}\Big(4 \ordv(2)+\ordv(\disc(F))\Big)\log \Nk(v).$$
\end{prop}

\begin{proof}
En \'etudiant des repr\'esentations du sous-groupe de
$2$-torsion de la jacobienne, M.Stoll (\cite{Stoll1}, th\'eor\`eme 6.1) a obtenu la minoration :
$$E_{v}(K)\geq \Big|2^{4}\disc(F)\Big|_{v}.$$
Ceci donne par un calcul direct la minoration annonc\'ee. Remarquons qu'il est possible d'utiliser ce r\'esultat de Stoll (\'ecrit pour des sextiques) en prenant l'un des coefficients $\beta_{j}$ \'egal \`a $0$ dans son paragraphe 3. Cela induit les m\^emes changements que pour les travaux de V. Flynn puisqu'il utilise le m\^eme plongement et les m\^emes matrices agissant sur $\mathbb{P}^{3}$.
\end{proof}

\subsection{Estimation aux places archim\'ediennes}

\subsubsection{Minoration de $\Lambda_{\Theta,v}(P)$}

On veut dans cette sous-partie minorer la hauteur locale archim\'edienne d\'efinie pour $P\in{A_{\Theta}}(\mathbb{C})$ par :
$$\Lambda_{\Theta,v}(P)=-\log\left(\Big|\theta_{a,b}(Z(P))\Big|_{v}\;e^{-\pi \,^{t}\Ima Z(\Ima\tau)^{-1}\Ima Z}\right),$$
o\`u on a fix\'e $[a,b]=[1/2,1/2,0,1/2]$. Pour tout vecteur $x=[x_{1},x_{2}]\in{\mathbb{R}^{2}}$, on d\'efinit la quantit\'e $\delta(x):=\min\{d(x_{1},\mathbb{Z}),d(x_{2},\mathbb{Z})\}$. 
Nous montrons tout d'abord une batterie de lemmes analytiques utiles pour l'estimation des fonctions th\^eta :

\begin{lem} \label{lem1}
Pour toute matrice $\tau\in{F_{2}}$ on a la minoration pour tout vecteur r\'eel $R=[R_{1},R_{2}]\in{\mathbb{R}^{2}}$:
$$\,^{t}R \Ima\tau R \geq (\Ima \tau_{1}-\Ima\tau_{12})\, R_{1}^{2}+(\Ima \tau_{2}-\Ima\tau_{12})\, R_{2}^{2}.$$
\end{lem}

\begin{proof}
Il suffit de d\'evelopper la forme quadratique et d'\'ecrire :

\begin{center}
\begin{tabular}{lll}

$\,^{t}R \Ima \tau R$ & $=$ & $R_{1}^{2}\Ima \tau_{1}+R_{2}^{2}\Ima \tau_{2}+2R_{1}R_{2}\Ima \tau_{12}$\\
\\

 & $\geq$ & $R_{1}^{2}\Ima \tau_{1}+R_{2}^{2}\Ima \tau_{2}-(R_{1}^{2}+R_{2}^{2})\Ima \tau_{12}$\\
\\

 & $\geq$ & $R_{1}^{2}(\Ima \tau_{1}-\Ima\tau_{12})+R_{2}^{2}(\Ima \tau_{2}-\Ima\tau_{12}).$\\
\\

\end{tabular}
\end{center}

On gardera en m\'emoire que $\Ima \tau_{2}\geq \Ima \tau_{1}\geq 2\Ima \tau_{12}\geq 0$ et $\Ima\tau_{1}>0$. Ces in\'egalit\'es impliquent que le minorant est une fonction d\'efinie positive.
\end{proof}

\begin{lem} \label{polaire}
Soit $r>0$ un r\'eel. Alors on a l'in\'egalit\'e :
$$\int_{r}^{+\infty}e^{-x^{2}}dx \leq \frac{\sqrt{\pi}}{2}e^{-r^{2}}.$$
\end{lem}

\begin{proof}
On se ram\`ene au calcul en coordonn\'ees polaires :
$$\left(\int_{r}^{+\infty}e^{-x^{2}}dx\right)^{2}=\int\!\!\!\int_{\substack{x\geq r\\y\geq r}}e^{-x^{2}-y^{2}}dxdy\leq \int\!\!\!\int_{\substack{\rho\geq r\sqrt{2}\\ \theta\in{[0,\frac{\pi}{2}]} }}e^{-\rho^{2}}\,\rho d\rho d\theta =\frac{\pi}{4}\,e^{-2r^{2}}.$$
\end{proof}

\begin{lem} \label{lem2}
Soient $\alpha > 0$ et $\beta \in{\mathbb{R}}$. Alors si $\beta \notin{\mathbb{Z}}$:
$$\sum_{n_{1}\in{\mathbb{Z}}}e^{-\alpha (n_{1}+\beta)^{2}}\leq \left(2+\frac{\sqrt{\pi}}{\sqrt{\alpha}}\right)e^{-\alpha d(\beta,\mathbb{Z})^{2}},$$ et si $\beta \in{\mathbb{Z}}$ :
$$\sum_{n_{1}\in{\mathbb{Z}}}e^{-\alpha (n_{1}+\beta)^{2}}\leq 1+\frac{\sqrt{\pi}}{\sqrt{\alpha}}.$$
De plus si $\beta=\frac{1}{2}$ on a :
$$\sum_{\substack{n_{1}\in{\mathbb{Z}}\\n_{1}\neq -3,-2,-1,0,1,2}}e^{-\alpha (n_{1}+\beta)^{2}}\leq \left(\frac{\sqrt{\pi}}{\sqrt{\alpha}}\right)e^{-\alpha\frac{25}{4}}.$$
Enfin si $\beta=0$ on a :
$$\sum_{\substack{n_{1}\in{\mathbb{Z}}\\n_{1}\neq -1,0,1}}e^{-\alpha n_{1}^{2}}\leq \left(\frac{\sqrt{\pi}}{\sqrt{\alpha}}\right)e^{-\alpha}.$$
\end{lem}

\begin{proof}
On d\'emontre la premi\`ere in\'egalit\'e, $\beta \notin{\mathbb{Z}}$, les autres s'en d\'eduisent.
On m\`ene une comparaison s\'erie-int\'egrale pour la fonction $f$ donn\'ee par $f(x)=e^{-\alpha(x+\beta)^{2}}$. On note $n_{0}$ le plus grand entier inf\'erieur \`a $-\beta$ (on notera $n_{0}=\lfloor -\beta \rfloor$) et on utilise la distance $d(\beta,\mathbb{Z})=\min\{|n_{0}+\beta|,|n_{0}+1+\beta|\}$. On obtient alors la majoration :

\begin{tabular}{lll}
$\displaystyle{\sum_{n_{1}\in{\mathbb{Z}}}e^{-\alpha(n_{1}+\beta)^{2}}}$ & $\leq$ & $\displaystyle{\int_{-\infty}^{n_{0}}e^{-\alpha(x+\beta)^{2}}dx +e^{-\alpha(n_{0}+\beta)^{2}}+e^{-\alpha(n_{0}+1+\beta)^{2}}+\int_{n_{0}+1}^{+\infty}e^{-\alpha(x+\beta)^{2}}dx,}$\\
\\
\end{tabular}

donc :
$$\sum_{n_{1}\in{\mathbb{Z}}}e^{-\alpha(n_{1}+\beta)^{2}}\leq \int_{(-n_{0}-\beta)\sqrt{\alpha}}^{+\infty}e^{-x^{2}}\frac{dx}{\sqrt{\alpha}} +2e^{-\alpha d(\beta,\mathbb{Z})^{2}}+\int_{(n_{0}+1+\beta)\sqrt{\alpha}}^{+\infty}e^{-x^{2}}\frac{dx}{\sqrt{\alpha}}.$$
Il suffit alors d'utiliser l'in\'egalit\'e du lemme \ref{polaire} pour conclure.
\end{proof}

\begin{lem} \label{lem3}
Soient $\alpha > 0$ et $\beta \in{\mathbb{R}}$. Alors :
$$\sum_{n_{1}\in{\mathbb{Z}}}\Big|n_{1}+\frac{1}{2}\Big|e^{-\alpha (n_{1}+\beta)^{2}}\leq C(\alpha,\beta)e^{-\alpha d(\beta,\mathbb{Z})^{2}},$$
o\`u l'on peut prendre :
$$C(\alpha,\beta)=\frac{1}{\alpha}+\Big|\beta-\frac{1}{2}\Big|\frac{\sqrt{\pi}}{\sqrt{\alpha}}+\frac{1}{2}\left(\sqrt{\left(\beta-\frac{1}{2}\right)^{2}+\frac{2}{\alpha}}+2\right)^{2}+\frac{1}{2}.$$
\end{lem}

\begin{proof}
On m\`ene ici une comparaison s\'erie-int\'egrale pour la fonction $f$ donn\'ee par $f(x)=|x+\frac{1}{2}|e^{-\alpha(x+\beta)^{2}}$. Il y a ici trois changements de sens de variation (car il y en a un en $-1/2$). On notera $x_{\max 1}<-\frac{1}{2}<x_{\max 2}$ les abscisses des trois maxima locaux. L'\'etude de la d\'eriv\'ee donne les expressions $x_{\substack{\max 1\\ \max2}}=\frac{1}{2}\left(-\beta-1/2\mp\sqrt{(\beta+1/2)^{2}-2\beta+2/\alpha}\right).$ Posons $\displaystyle{N_{1}=\lfloor x_{\max 1} \rfloor}$ et $\displaystyle{N_{2}=\lfloor x_{\max 2} \rfloor}$, o\`u $\lfloor x \rfloor$ d\'esigne la partie enti\`ere de $x$. On a alors la majoration :
$$\sum_{n_{1}\in{\mathbb{Z}}}\Big|n_{1}+\frac{1}{2}\Big|e^{-\alpha(n_{1}+\beta)^{2}}\leq A + B + C ,$$
o\`u :
$$A=\int_{-\infty}^{N_{1}}\Big|x+\frac{1}{2}\Big|e^{-\alpha(x+\beta)^{2}}dx,\;\;B=\sum_{n_{1}=N_{1}}^{N_{2}+1}\Big|n_{1}+\frac{1}{2}\Big|e^{-\alpha(n_{1}+\beta)^{2}},\;\;C=\int_{N_{2}+1}^{+\infty}\Big|x+\frac{1}{2}\Big|e^{-\alpha(x+\beta)^{2}}dx.$$

Alors en posant $r=(-N_{1}-\beta)\sqrt{\alpha}>0$ on obtient par in\'egalit\'e triangulaire :
$$\displaystyle{A=\int_{r}^{+\infty}\Big|\frac{x}{\sqrt{\alpha}}+\beta-\frac{1}{2}\Big|e^{-x^{2}}\frac{dx}{\sqrt{\alpha}}}\leq \frac{1}{\alpha}\int_{r}^{+\infty}xe^{-x^{2}}dx \;+ \; \frac{|\beta-\frac{1}{2}|}{\sqrt{\alpha}}\int_{r}^{+\infty}e^{-x^{2}}dx,$$
donc par int\'egration directe et par l'in\'egalit\'e du lemme \ref{polaire} :
$$A \leq \frac{1}{2\alpha}e^{-\alpha(-N_{1}-\beta)^{2}}+ \frac{|\beta-\frac{1}{2}|\sqrt{\pi}}{2\sqrt{\alpha}}e^{-\alpha(-N_{1}-\beta)^{2}}\leq \left(\frac{1}{2\alpha}+\frac{|\beta-\frac{1}{2}|\sqrt{\pi}}{2\sqrt{\alpha}} \right)e^{-\alpha d(\beta,\mathbb{Z})^{2}}.$$

On obtient la m\^eme majoration pour le terme $C$. Reste le terme m\'edian :
$$B \leq \sum_{n_{1}=N_{1}}^{N_{2}+1}\Big|n_{1}+\frac{1}{2}\Big|e^{-\alpha d(\beta,\mathbb{Z})^{2}}\leq \Big(\sum_{n_{1}=0}^{N_{2}+1}\Big(n_{1}+\frac{1}{2}\Big)-\sum_{n_{1}=N_{1}}^{-1}\Big(n_{1}+\frac{1}{2}\Big) \Big)e^{-\alpha d(\beta,\mathbb{Z})^{2}},$$
donc :
$$Be^{\alpha d(\beta,\mathbb{Z})^{2}} \leq \frac{N_{1}^{2}+N_{2}^{2}+4N_{2}+4}{2} =\frac{(N_{2}-N_{1})^{2}+2N_{1}N_{2}+4N_{2}+4}{2},$$
donc en utilisant les in\'egalit\'es $N_{1}\leq -1$ et $2\leq -2N_{1}$ :
$$Be^{\alpha d(\beta,\mathbb{Z})^{2}} \leq \frac{(N_{2}-N_{1})^{2}+2(N_{2}-N_{1})+2}{2}=\frac{(N_{2}-N_{1}+1)^{2}+1}{2}.$$

Or $0\leq N_{2}-N_{1}<x_{\mathrm{max2}}-x_{\mathrm{max1}}+1=\sqrt{(\beta-\frac{1}{2})^{2}+\frac{2}{\alpha}}+1$, donc :
$$B \leq \Big(\frac{1}{2}\Big(\sqrt{\Big(\beta-\frac{1}{2}\Big)^{2}+\frac{2}{\alpha}}+2 \Big)^{2}+\frac{1}{2}\Big)e^{-\alpha d(\beta,\mathbb{Z})^{2}}.$$

Il suffit alors de r\'eunir les majorations des termes $A,B,C$ pour obtenir le lemme.
\end{proof}

\begin{lem} \label{lem4}
On pose $\delta(a+Y)=\min\{d(\frac{1}{2}+y_{1},\mathbb{Z}),d(\frac{1}{2}+y_{2},\mathbb{Z})\}$. On suppose que $||(X,Y)||\leq \frac{1}{2}$.
On a alors la majoration pour tout $u\in{[0,1]}$ :
$$\sum_{n\in{\mathbb{Z}^{2}}}\Big|n_{i}+\frac{1}{2}\Big| e^{-\pi \, ^{t}(n+a+Yu) \Ima \tau (n+a+Yu)}\leq C_{2}(y_{i})e^{-\pi(\Tr(\Ima \tau)-2\Ima\tau_{12})\delta(a+Y)^{2}},$$
o\`u l'on peut prendre : $C_{2}(y_{i})=4\Big(\frac{4}{\pi}+2|y_{i}|+\frac{1}{2}\Big(\sqrt{y_{i}^{2}+\frac{8}{\pi}}+2\Big)^{2}+\frac{1}{2}\Big).$
\end{lem}

\begin{proof}
On applique successivement les lemmes \ref{lem1}, \ref{lem2} et \ref{lem3} en sp\'ecialisant $\alpha = \pi(\Ima \tau_{i}-\Ima\tau_{12})$ et $\beta = \frac{1}{2}+uy_{i}$ (avec $u\in{[0,1]}$) pour $i\in{\{1,2\}}$. Le maximum sur $u$ est atteint, pour le majorant, en $u=1$, car $|y_{i}|\leq \frac{1}{2}$.
\end{proof}

Nous pouvons donc \`a pr\'esent montrer la proposition suivante :
\begin{prop}\label{archim}
Soit $C/k$ une courbe de genre 2 et soit $v$ une place archim\'edienne. Soit $P$ un point de
$\Jac(C)(\bar{k}_{v})\cong \mathbb{C}^{2}/(\mathbb{Z}^{2}+\tau_{v}
\mathbb{Z}^{2})$ hors du support du diviseur $\Theta$. On note $Z=X+\tau_{v} Y$ une coordonn\'ee de $P$, avec $Y=[y_{1},y_{2}]$. On d\'efinit la norme de vecteur $||(X,Y)||=\max\{|x_{1}|,|x_{2}|,|y_{1}|,|y_{2}|\}$. Alors
la hauteur locale archim\'edienne $\Lambda_{\Theta,v}$ peut \^etre minor\'ee de la façon suivante, d\`es que $||(X,Y)||\leq \frac{1}{2}$ :

\begin{center}
\begin{tabular}{lll}

$\Lambda_{\Theta,v}(P)$ & $\geq$ &$\displaystyle{\pi(\Tr(\Ima \tau_{v})-2\Ima\tau_{v,12})\delta(a+Y)^{2}-\log\Big(4+\frac{3}{2}\Tr(\Ima \tau_{v})\Big)}$\\
\\

 &  & $\displaystyle{+\log \frac{1}{||(X,Y)||}-\log C_{3}(Y),}$\\
\\
\end{tabular}
\end{center}
o\`u l'on peut prendre : $C_{3}(Y)=\max_{i\in\{1,2\}}8\pi \Big(\frac{4}{\pi}+2|y_{i}|+\frac{1}{2}\Big(\sqrt{y_{i}^{2}+\frac{8}{\pi}}+2\Big)^{2}+\frac{1}{2}\Big).$
\end{prop}

\begin{Remarque}
On a la majoration $C_{3}(Y)\leq
239,2$ pour $y_{i}\leq 1/2$.
\end{Remarque}

\begin{proof}
On notera tout au long de la preuve : $\tau=\tau_{v}$.
Calculons, pour $Z=X+\tau Y$, avec $X$ et $Y$ des vecteurs de $\mathbb{R}^{2}$ : 
$$\, ^{t}\Ima Z(\Ima \tau)^{-1}\Ima Z=\, ^{t}\Ima (\tau Y)(\Ima \tau)^{-1}\Ima (\tau Y)=\, ^{t}Y \Ima \tau Y.$$
Posons : $$\zeta_{n}(X,Y):=2i\pi \left(\frac{1}{2}\, ^{t}(n+a)\tau(n+a)+\, ^{t}(n+a)\tau Y+\, ^{t}(n+a)(X+b)\right),$$
et : $$g(X,Y):=\theta_{a,b}(X+\tau Y)=\sum_{n\in{\mathbb{Z}^{2}}}e^{\zeta_{n}(X,Y)}.$$
On veut donc majorer la quantit\'e : $|g(X,Y)|e^{-\pi \, ^{t}Y \Ima \tau Y}.$ Tout d'abord par l'in\'egalit\'e des accroissements finis, avec
$||(X,Y)||=\sup\{|x_{1}|,|x_{2}|,|y_{1}|,|y_{2}|\}$ et $|||.|||$ la norme subordonn\'ee :
$$|g(X,Y)-g(0,0)|\leq \left(\max_{(X',Y')\in{[(0,0),(X,Y)]}}|||dg_{|(X',Y')}|||\right)||(X,Y)-(0,0)||.$$
donc comme $g(0,0)=\theta_{a,b}(0)=0$ :
$$|g(X,Y)|\leq \left(\max_{u\in{[0,1]}}|||dg_{|u(X,Y)}|||\right)||(X,Y)||.$$
On a alors en \'ecrivant $[X,Y]=[x_{1}, \; x_{2}, \; y_{1}, \; y_{2}]$ :

\begin{center}
\begin{tabular}{lll}

$\displaystyle{\frac{\partial g}{\partial x_{1}}(X,Y)}$ & $=$ & $\displaystyle{\sum_{n\in{\mathbb{Z}^{2}}}2\pi i\left(n_{1}+\frac{1}{2}\right)e^{\zeta_{n}(X,Y)}},$\\

$\displaystyle{\frac{\partial g}{\partial x_{2}}(X,Y)}$ & $=$ & $\displaystyle{\sum_{n\in{\mathbb{Z}^{2}}}2\pi i\left(n_{2}+\frac{1}{2}\right)e^{\zeta_{n}(X,Y)}},$\\

$\displaystyle{\frac{\partial g}{\partial y_{1}}(X,Y)}$ & $=$ & $\displaystyle{\sum_{n\in{\mathbb{Z}^{2}}}2\pi i\left(\left(n_{1}+\frac{1}{2}\right)\tau_{1}+\left(n_{2}+\frac{1}{2}\right)\tau_{12}\right)e^{\zeta_{n}(X,Y)}},$\\

$\displaystyle{\frac{\partial g}{\partial y_{2}}(X,Y)}$ & $=$ & $\displaystyle{\sum_{n\in{\mathbb{Z}^{2}}}2\pi i\left(\left(n_{1}+\frac{1}{2}\right)\tau_{12}+\left(n_{2}+\frac{1}{2}\right)\tau_{2}\right)e^{\zeta_{n}(X,Y)}}.$\\

\end{tabular}
\end{center}

De plus :

\begin{center}
\begin{tabular}{lll}

$\Ree(\zeta_{n}(Xu,Yu))$ & $=$ & $-\pi \, ^{t}(n+a)\Ima \tau (n+a)-2\pi \, ^{t}(n+a)\Ima \tau Yu$\\
\\

 & $=$ & $-\pi \, ^{t}(n+a+Yu)\Ima \tau (n+a+Yu)+\pi u^{2}\, ^{t}Y \Ima \tau Y.$

\end{tabular}
\end{center}

On obtient alors pour tout vecteur $(X,Y)$ non nul :

\begin{center}
\begin{tabular}{llll} 
$\displaystyle{\frac{|g(X,Y)|}{2\pi ||(X,Y)||}}$ & $\leq$ & $\displaystyle{\max_{u\in{[0,1]}}}$ & $\displaystyle{\Big[\sum_{n\in{\mathbb{Z}^{2}}}\Big(\Big|n_{1}+\frac{1}{2}\Big|+\Big|n_{2}+\frac{1}{2}\Big|\Big)e^{\Ree(\zeta_{n}(Xu,Yu))}}$\\
\\
&  &  & $\displaystyle{+ \sum_{n\in{\mathbb{Z}^{2}}}\left(\Big|n_{1}+\frac{1}{2}\Big||\tau_{1}|+\Big|n_{2}+\frac{1}{2}\Big||\tau_{12}|\right)e^{\Ree(\zeta_{n}(Xu,Yu))}}$\\
\\
&  &  & $\displaystyle{+ \sum_{n\in{\mathbb{Z}^{2}}}\left(\Big|n_{1}+\frac{1}{2}\Big||\tau_{12}|+\Big|n_{2}+\frac{1}{2}\Big||\tau_{2}|\right)e^{\Ree(\zeta_{n}(Xu,Yu))}\Big]}.$\\

\end{tabular}
\end{center}

On obtient ainsi :

\begin{tabular}{lll}

$\displaystyle{\frac{|g(X,Y)|e^{-\pi \, ^{t}Y \Ima \tau Y}}{2\pi ||(X,Y)||}}$ & $\leq$ & $\displaystyle{\max_{u\in{[0,1]}}\Big[\sum_{n\in{\mathbb{Z}^{2}}}\Big((1+|\tau_{1}|+|\tau_{12}|)\Big|n_{1}+\frac{1}{2}\Big|}+$\\
\\

&  & $\displaystyle{(1+|\tau_{2}|+|\tau_{12}|)\Big|n_{2}+\frac{1}{2}\Big|\Big)e^{\Ree(\zeta_{n}(Xu,Yu))-\pi \, ^{t}Y \Ima \tau Y} \Big]},$\\
\\
\end{tabular}

donc :

\begin{tabular}{lll}

$\displaystyle{\frac{|g(X,Y)|e^{-\pi \, ^{t}Y \Ima \tau Y}}{2\pi ||(X,Y)||}}$ & $\leq$ & $\displaystyle{\max_{u\in{[0,1]}}\Big[\sum_{n\in{\mathbb{Z}^{2}}}\Big((1+|\tau_{1}|+|\tau_{12}|)\Big|n_{1}+\frac{1}{2}\Big|}+\hspace{3cm}$\\
\\

&  & $\displaystyle{(1+|\tau_{2}|+|\tau_{12}|)\Big|n_{2}+\frac{1}{2}\Big|\Big)e^{-\pi \, ^{t}(n+a+Yu) \Ima \tau (n+a+Yu)}\Big]}.$ \\
\\
\end{tabular}

On utilise alors le lemme \ref{lem4} dans la derni\`ere majoration :

\begin{tabular}{lll}

$\displaystyle{\frac{|g(X,Y)|e^{-\pi \, ^{t}Y \Ima \tau Y}}{2\pi ||(X,Y)||}}$ & $\leq$ & $\displaystyle{(1+|\tau_{1}|+|\tau_{12}|)C_{2}(y_{1})}e^{-\pi(\Tr(\Ima \tau)-2\Ima\tau_{12})\delta(a+Y)^{2}}$\\
\\

&  & $\displaystyle{+ (1+|\tau_{2}|+|\tau_{12}|)C_{2}(y_{2})}e^{-\pi(\Tr(\Ima \tau)-2\Ima\tau_{12})\delta(a+Y)^{2}},$\\
\end{tabular}

\vspace{0.1cm}

donc on obtient en notant $C_{2}(Y):=\max\{C_{2}(y_{1}),C_{2}(y_{2})\}$ :
$$
\frac{|g(X,Y)|e^{-\pi \, ^{t}Y \Ima \tau Y}}{2\pi ||(X,Y)||}\leq \displaystyle{(2+|\tau_{1}|+|\tau_{2}|+2|\tau_{12}|)}C_{2}(Y)e^{-\pi(\Tr(\Ima \tau)-2\Ima\tau_{12})\delta(a+Y)^{2}}.
$$
En prenant l'oppos\'e du logarithme de cette derni\`ere in\'egalit\'e il vient finalement :

\begin{center}
\begin{tabular}{lll}

$\Lambda_{\Theta,v}(P)$ & $\geq$ &$\displaystyle{\pi(\Tr(\Ima \tau)-2\Ima\tau_{12})\delta(a+Y)^{2}-\log(2+|\tau_{1}|+|\tau_{2}|+2|\tau_{12}|)}$\\
\\

&  & $\displaystyle{+\log \frac{1}{||(X,Y)||}-\log 2\pi C_{2}(Y),}$\\
\end{tabular}
\end{center}

De plus, en utilisant $|\tau_{i}|\leq \frac{1}{2}+\Ima\tau_{i}$ et $|\tau_{12}|\leq\frac{1}{2}+\frac{1}{2}\Ima\tau_{i}$ pour $i=1$ et $i=2$ on obtient :
$$\log(2+|\tau_{1}|+|\tau_{2}|+2|\tau_{12}|)\leq \log\Big(4+\frac{3}{2}\Tr(\Ima \tau)\Big).$$

Ceci ach\`eve la preuve de la proposition \ref{archim}.
\end{proof}

\subsubsection{Minoration de $|\Delta(\tau_{v})|$}

On va donner dans cette section une minoration de la norme des constantes th\^eta paires en dimension 2. Soit $v$ une place archim\'edienne. On note ici $\tau=\tau_{v}$. Comme $\Delta(\tau)$ s'annule uniquement en $\tau_{12}=0$ (pour $\tau$ dans $F_2$, voir par exemple \cite{Kling}, proposition 2 page 115), on s'attend \`a voir appara\^itre une condition sur l'espace de modules des surfaces ab\'eliennes principalement polaris\'ees. 

Il y a dix constantes th\^eta $\theta_{ab}(0,\tau)$ non nulles en dimension 2 ; elles correspondent exactement aux caract\'eristiques paires :
$$\left[\begin{array}{c}
a \\
b 	
\end{array}\right] \in{\theta_{1}=\left\{ \left[\begin{array}{c}
0 \\
0 \\
0 \\
0 	
\end{array}\right], \left[\begin{array}{c}
0 \\
0 \\
0 \\
1/2 	
\end{array}\right],\left[\begin{array}{c}
0 \\
0 \\
1/2 \\
0 	
\end{array}\right],\left[\begin{array}{c}
0 \\
0 \\
1/2 \\
1/2 	
\end{array}\right]\right\}},$$

$$\left[\begin{array}{c}
a \\
b 	
\end{array}\right] \in{\theta_{2}=\left\{\left[\begin{array}{c}
1/2 \\
0 \\
0 \\
0 	
\end{array}\right],\left[\begin{array}{c}
0 \\
1/2 \\
0 \\
0 	
\end{array}\right],\left[\begin{array}{c}
1/2 \\
1/2 \\
0 \\
0 	
\end{array}\right],\left[\begin{array}{c}
0 \\
1/2 \\
1/2 \\
0 	
\end{array}\right],\left[\begin{array}{c}
1/2 \\
0 \\
0 \\
1/2 	
\end{array}\right],\left[\begin{array}{c}
1/2 \\
1/2 \\
1/2 \\
1/2 	
\end{array}\right] \right\}}.$$

On a donc $\mathcal{Z}_{2}=\theta_{1}\cup\theta_{2}$.

On rappelle la relation :
$$\Big|\theta_{ab}(0,\tau)\Big|=\Big|\theta_{00}(\tau a +b,\tau)e^{i\pi \,^{t}a\tau a+2i\pi\,^{t}ab}\Big|=\Big|\theta_{00}(\tau a +b,\tau)\Big|e^{-\pi \,^{t}a\Ima \tau a}.$$

Si on pose $Q_{a,b}(n)=^{t}(n+a)\tau(n+a)+2\, ^{t}(n+a)(b)$, on a de plus :
$$\theta_{a,b}(0,\tau)=\sum_{n\in{\mathbb{Z}^{2}}}e^{i\pi Q_{a,b}(n)}.$$

\begin{lem}\label{carac generique}
Soit $(a,b)\in{\mathcal{Z}_{2}}$. Soit $T_{a,b}=\Big\{\displaystyle{n\in{\mathbb{Z}^{2}}\,|\,\Ima Q_{a,b}(n)=\min_{m\in{\mathbb{Z}^{2}}}\Ima Q_{a,b}(m)}\Big\}$. On a la propri\'et\'e :
$$\forall n,n'\in{T_{a,b}},\; e^{i\pi Q_{a,b}(n)}=e^{i\pi Q_{a,b}(n')}.$$ De plus :
$$\Big|\theta_{a,b}(0,\tau) \Big|\geq 2\Card(T_{a,b})e^{-\pi \min_{m\in{\mathbb{Z}^{2}}}\Ima Q_{a,b}(m)}-\sum_{n\in{\mathbb{Z}^{2}}}e^{-\pi\Ima Q_{a,b}(n)}.$$

\end{lem}

\begin{proof}
La propri\'et\'e de l'\'enonc\'e se v\'erifie directement sur les dix couples $(a,b)\in{\mathcal{Z}_{2}}$.
En effet remarquons tout d'abord que $T_{a,b}$ est un ensemble fini pour $(a,b)\in{\mathcal{Z}_{2}}$. Il est de cardinal 1 lorsque $a=0$ et de cardinal 2 sinon. La propri\'et\'e de cet ensemble se v\'erifie alors directement en calculant $Q_{a,b}(n)$ pour $n\in{\{-1,0\}^{2}}$. L'in\'egalit\'e triangulaire donne ensuite :
$$\Big|\theta_{a,b}(0,\tau) \Big|\geq \Big|\sum_{n\in{T_{a,b}}}e^{i\pi Q_{a,b}(n)}\Big|-\Big|\sum_{n\in{\mathbb{Z}^{2}\backslash T_{a,b}}}e^{i\pi Q_{a,b}(n)}\Big|.$$
Le choix de $T_{a,b}$ implique $\displaystyle{\sum_{n\in{T_{a,b}}}e^{i\pi Q_{a,b}(n)}=\Card(T_{a,b})e^{i\pi Q_{a,b}(m)}}$ pour un $m$ quelconque choisi dans $T_{a,b}$. On obtient alors directement l'in\'egalit\'e annonc\'ee en utilisant : $$\sum_{n\in{\mathbb{Z}^{2}}\backslash T_{a,b}}e^{-\pi\Ima Q_{a,b}(n)}+\sum_{n\in{T_{a,b}}}e^{-\pi\Ima Q_{a,b}(n)}=\sum_{n\in{\mathbb{Z}^{2}}}e^{-\pi\Ima Q_{a,b}(n)}.$$
\end{proof}

\begin{prop}\label{propcarac00}
Pour les caract\'eristiques $[a,b]\in{\theta_{1}}$ (donc v\'erifiant $a=0$) on a
la minoration, valable pour tout $\tau\in{F_2}$ :
$$\Big|\theta_{ab}(0,\tau)\Big|\geq f_{1}(\tau),$$
o\`u on a pos\'e $$f_{1}(\tau)=1-\left(4e^{-\pi \sqrt{3}/2}+\sum_{n_1^2+n_2^2>1}e^{-\pi\frac{\sqrt{3}}{4}(n_1^2+n_2^2)}\right).$$ 

Une estimation directe donne $f_{1}(\tau)\geq 0,44$.
\end{prop}

\begin{proof}
On utilise le lemme \ref{carac generique}. On a dans ce cas $\displaystyle{\min_{m\in{\mathbb{Z}^{2}}}\Ima Q_{a,b}(m)=0}$ et $T_{a,b}=\{(0,0)\}$. Il suffit alors d'utiliser \cite{Kling} page 116.

\end{proof}

\begin{prop}\label{propcarac1000}
Soit $[a,b]=[1/2,0,0,0]$ ou $[a,b]=[1/2,0,0,1/2]$. On a la
minoration, valable pour tout $\tau\in{F_2}$ :

$$\Big|\theta_{ab}(0,\tau)\Big|e^{\pi\,^{t}a\Ima \tau a}\geq f_{2}(\tau),$$
o\`u on a pos\'e
$$f_{2}(\tau)=2-\left(\sum_{n\in{\mathbb{Z}^2}\backslash\{(0,0),(-1,0)\}}e^{-\pi\frac{\sqrt{3}}{4}(n_1(n_1+1)+n_2^2)}\right).$$

De plus on d\'eduit le minorant pour les caract\'eristiques $[a,b]=[0,1/2,0,0]$ et $[a,b]=[0,1/2,1/2,0]$ en permutant les coordonn\'ees dans cette derni\`ere expression. 

Une estimation directe donne $f_{2}(\tau)\geq 0,75$.
\end{prop}

\begin{proof}
On fait le calcul pour la caract\'eristique $[1/2,0,0,0]$, le deuxi\`eme calcul se d\'eduit du premier en changeant $n_{1}$ en $n_{2}$. 

On a ici : $\displaystyle{\min_{m\in{\mathbb{Z}^{2}}}\Ima Q_{a,b}(m)= ^{t}a\Ima\tau a}$ et $T_{a,b}=\{(0,0),(-1,0)\}$. On utilise le lemme \ref{carac generique} et \cite{Kling} page 117.

\end{proof}

\begin{lem}\label{exp complexe}

Soit $z$ un nombre complexe v\'erifiant $\mathrm{Im} z\geq 0$ et $\vert \mathrm{Re} z\vert \leq \frac{1}{2}$. Alors on a les in\'egalit\'es : 

\begin{enumerate}
\item $\displaystyle{\vert e^{i\pi z}+1\vert\geq 1,}$
\item $\displaystyle{\vert e^{i\pi z}-1\vert\geq 0,28\min\{1, \; \pi\vert z\vert\}.}$
\end{enumerate}

\end{lem}

\begin{proof}
Le premi\`ere in\'egalit\'e est imm\'ediate. Pour la seconde, on peut commencer par supposer $\pi\vert z\vert<1$. Dans ce cas nous avons $$\vert e^{i\pi z}-1\vert=\vert i\pi z+\sum_{k=2}^{\infty}\frac{(i\pi z)^k}{k!}\vert\geq \pi \vert z\vert -\sum_{k=2}^{\infty}\frac{(\pi \vert z\vert)^k}{k!}=\pi \vert z\vert(2+\frac{1-e^{\pi\vert z\vert}}{\pi \vert z\vert})\geq \pi\vert z\vert(3-e)\geq 0,28\pi\vert z\vert.$$ \`A pr\'esent si $\pi\vert z\vert\geq 1$, cela implique que $M=\max\{\pi \Ima z,\; \pi\mathrm{Re} z\}\geq 1/\sqrt{2}$. Si $M=\pi \Ima z$, on calcule alors $$\vert e^{i\pi z}-1\vert^2=1+e^{-2\pi\Ima z}-2e^{-\pi\Ima z}\cos(\mathrm{Re} z\, \pi)\geq 1+e^{-2\pi\Ima z}-2e^{-\pi\Ima z}\geq 1+e^{-\pi\sqrt{2}}-2e^{-\pi/\sqrt{2}}.$$Si $M=\pi \mathrm{Re} z$, on \'ecrit $$\vert e^{i\pi z}-1\vert^2=1+e^{-2\pi\Ima z}-2e^{-\pi\Ima z}\cos(\mathrm{Re} z\, \pi)\geq 1+e^{-2\pi\Ima z}-2e^{-\pi\Ima z}\cos(\frac{1}{\sqrt{2}})\geq 1-\cos^2({\frac{1}{\sqrt{2}}}).$$

On peut donc minorer, dans le cas o\`u $\pi \vert z\vert \geq1$ :

$$\vert e^{i\pi z}-1\vert^2\geq \min\{1+e^{-\pi\sqrt{2}}-2e^{-\pi/\sqrt{2}},\,1-\cos^2({\frac{1}{\sqrt{2}}})\}\geq 0,25.$$
\end{proof}

\begin{prop}\label{propcarac1100}
Soit $[a,b]=[1/2,1/2,\varepsilon/2,\varepsilon/2]$, avec $\varepsilon\in{\{0,1\}}$. On a la minoration :

$$\Big|\theta_{ab}(0,\tau)\Big|e^{\pi\,(^{t}a\Ima \tau a-\Ima \tau_{12})}\geq f_{3,\varepsilon}(\tau),$$
o\`u on a pos\'e
$$f_{3,\varepsilon}(\tau)=2\vert1+(-1)^\varepsilon e^{\pi i\tau_{12}}\vert\left(2-\left(\sum_{m=0}^{+\infty}e^{-\frac{\pi\sqrt{3}}{4}m(m+1)}(2m+1)\right)^2 \right).$$

\end{prop}

\begin{Remarque} Pour que le minorant soit strictement positif lorsque $\varepsilon=1$, on doit donc imposer $\tau_{12}\neq 0$, ce qui souligne le fait qu'un produit de deux courbes elliptiques est un cas d\'eg\'en\'er\'e de vari\'et\'e ab\'elienne principalement polaris\'ee de dimension 2. 
\end{Remarque}

\begin{proof} On proc\`ede comme pour la proposition pr\'ec\'edente. On a ici : $$\displaystyle{\min_{m\in{\mathbb{Z}^{2}}}\Ima Q_{a,b}(m)= ^{t}a\Ima\tau a}-\Ima\tau_{12}\;\quad \mathrm{et}\;\quad T_{a,b}=\{(0,-1),(-1,0)\}.$$ On obtient le r\'esultat par l'application du lemme \ref{carac generique} et de \cite{Kling} page 117.

\end{proof}

Il ne reste plus qu'\`a r\'eunir les calculs pr\'ec\'edents :

\begin{prop}\label{minoration delta}
Soit $\tau\in{F_{2}}$. On a la minoration : $$\Big|\Delta(\tau)\Big|\geq c_0\min\{1, \pi\vert \tau_{12}\vert\}^2\; e^{-2\pi(\Tr(\Ima\tau)-\Ima\tau_{12})},$$ et on peut prendre $c_0=5.10^{-9}$. 
\end{prop}

\begin{proof}
On commence par \'ecrire $\displaystyle{\Delta(\tau)=2^{-12}\prod_{m\in{\mathcal{Z}_{2}}}\theta_{m}(0,\tau)^{2}}$ puis en utilisant les notations des propositions \ref{propcarac00}, \ref{propcarac1000} et \ref{propcarac1100} on a la minoration :
$$\Big|\Delta(\tau)\Big|\geq2^{-12}f_{1}(\tau)^{8}f_{2}(\tau)^{8}f_{3,0}(\tau)^{2}f_{3,1}(\tau)^{2}e^{-2\pi(\Tr(\Ima\tau)-\Ima\tau_{12})}.$$ 
Il suffit ensuite d'utiliser le lemme \ref{exp complexe} et les estimations num\'eriques pr\'ec\'edentes.
\end{proof}

\section{Minoration globale de la hauteur de N\'eron-Tate}

On montre dans cette partie comment \`a partir des informations locales on peut obtenir un th\'eor\`eme global de minoration de la hauteur de N\'eron-Tate sur une jacobienne de dimension 2.

\subsection{Lemme de z\'eros et principe des tiroirs}

\begin{lem}\label{point hors de la courbe}

Soit $k$ un corps de nombres. Soit $C/k$ une courbe de genre 2 plong\'ee dans $A$ sa jacobienne. Soient $P_{1}$ et $P_{2}$ des points non nuls de $A(k)$ tels que $P_{1}+ P_{2}\neq 0$. Alors :
$$\{\pm P_{1},\,\pm P_{2},\,\pm (P_{1}+P_{2})\}\nsubseteq C(k). $$
\end{lem}

\begin{proof}
La preuve propos\'ee ici montre un r\'esultat un peu plus g\'en\'eral. Soient $S_{1}=\{T_{1},...,T_{r}\}$ et $S_{2}=\{Q_{1},...,Q_{r}\}$ deux ensembles de points de $A(k)$ \`a exactement $r\geq 2$ \'el\'ements. On suppose que $S_{1}+S_{2}\subset C(k)$.

Posons alors $\displaystyle{C^{(1)}=\bigcap_{t\in{S_{1}}}(C-t)}$. C'est une sous-vari\'et\'e non vide et stricte de $A$, sa dimension vaut donc $0$ ou $1$. Or si $t\neq O$, l'ensemble $C\cap (C-t)$ est fini (cela vient du fait que $\Theta$ est associ\'e \`a une polarisation principale) ; donc la dimension de $C^{(1)}$ est z\'ero car $\Card(S_{1})\geq 2$. Comme de plus $S_{2}\subset C^{(1)}$ par construction, il vient :
$$r=\Card(S_{2})\leq \deg(C^{(1)})\leq 2,$$
la derni\`ere in\'egalit\'e \'etant justifi\'ee par le fait que $C\!\cdot\! C=2!=2$ (comme auto-intersection de diviseur).

On obtient alors le lemme en prenant $S_{1}=\{O,P_{1},-P_{2}\}$ et $S_{2}=\{O,-P_{1},P_{2}\}$ : on sait que $r=3$ dans ce cas grâce aux hypoth\`eses sur $P_{1}$ et $P_{2}$, il vient donc par contrapos\'ee : $$S_{1}+S_{2}=\{O,P_{1},P_{2},P_{1}+P_{2},-P_{1},-P_{2},-P_{1}-P_{2}\}\nsubseteq C(k).$$  Il suffit de remarquer que $O\in{C(k)}$ pour conclure. 
\end{proof}

\begin{Remarque} On peut d\'eduire de ce lemme une nouvelle preuve de la propri\'et\'e \ref{3-torsion}. Soit $Q$ un point d'ordre 3 exactement. On pose dans le lemme pr\'ec\'edent $S_{1}=S_{2}=\{O,Q,-Q\}$. Alors le lemme permet d'affirmer :
$$\{O,Q,-Q\}+\{O,Q,-Q\}=\{O,\pm Q,\pm [2]Q\}=\{O,\pm Q\}\nsubseteq C(k),$$
ce qui permet de conclure : $Q\notin C(k)$.
\end{Remarque}

\begin{prop}\label{minoration infinie}
Soit $k$ un corps de nombres, on pose
$m=|M_{k}^{\infty}|$. Soit $C/k$ une courbe de genre $2$, on note $A=\Jac(C)$ sa
jacobienne. Soit $M>2$ un r\'eel. Soit $P\in{A(k)}$ un point tel que ses multiples $\{[n]P,n\in{ \llbracket 0,2M^{4m} \rrbracket }\}$ soient tous distincts. Il vient alors : 
$$\hspace{-7cm}\exists n\in{\llbracket 0,2M^{4m} \rrbracket}, \; [n]P\notin \Theta,\, \; \forall v\in{M_{k}^{\infty}},$$
$$\Lambda_{\Theta,v}([n]P)\geq (1/2-1/M)^{2}\pi \Big(\Tr(\Ima \tau_{v})-2\Ima\tau_{12,v}\Big)-\log\Big(4+\frac{3}{2}\Tr(\Ima\tau_{v})\Big)+\log\frac{M}{240}.$$
\end{prop}

\begin{proof}
On a les applications :
$$A(\bar{k}_{v})\longrightarrow \mathbb{C}^{2}/(\mathbb{Z}^{2}+\tau_{v}\mathbb{Z}^{2}) \longrightarrow  \mathbb{R}^{2}/\mathbb{Z}^{2}\times \mathbb{R}^{2}/\mathbb{Z}^{2}$$
$$P \longmapsto Z_{v}(P)=X_{v}(P)+\tau_{v}Y_{v}(P) \longmapsto (X_{v}(P),Y_{v}(P)).$$

Soit alors l'application $F:A(k)\longrightarrow (\mathbb{R}/\mathbb{Z})^{4m}$ d\'efinie par : $$F(P)=(X_{v}(P),Y_{v}(P))_{v\in{M_{k}^{\infty}}}.$$
On divise alors $(\mathbb{R}/\mathbb{Z})^{4m}$ en $M^{4m}$ bo\^ites de taille $\frac{1}{M}$. On consid\`ere alors l'ensemble $\{F([n]P),n\in{\llbracket 0,2M^{4m} \rrbracket}\}$ : il contient $2M^{4m}+1$ points \`a r\'epartir dans $M^{4m}$ bo\^ites. Par le principe des tiroirs il existe donc trois entiers $n_{1}$, $n_{2}$ et $n_{3}$ tels que, avec $i>j$ : 
$$0\leq n_{1}<n_{2}<n_{3}\leq 2M^{4m},\;\; n_{i}-n_{j}\leq 2M^{4m},$$ 
$$ ||X_{v}([n_{i}-n_{j}]P)||\leq \frac{1}{M},\;\;\; ||Y_{v}([n_{i}-n_{j}]P)||\leq \frac{1}{M}.$$

Posons $P_{1}=[n_{3}-n_{2}]P$ et $P_{2}=[n_{2}-n_{1}]P$. Alors $P_{1}+P_{2}=[n_{3}-n_{1}]P$. En appliquant le lemme \ref{point hors de la courbe}, on sait que dans l'ensemble de points $\{P_{1},P_{2},P_{1}+P_{2},-P_{1},-P_{2},-P_{1}-P_{2}\}$ il y en a au moins un qui n'est pas sur le diviseur $\Theta$. C'est ce point qu'on choisit : on le note $[n]P$ (ou peut-\^etre $[n](-P)$, le fait de prendre \'eventuellement l'oppos\'e n'est pas g\^enant car la hauteur locale est paire).

On a $||(X,Y)||\leq 1/M$ et $d(1/2+y_{i},\mathbb{Z})^{2}\geq(1/2-1/M)^{2}$. On obtient alors, en reportant ces approximations dans la proposition \ref{archim} :
$$\Lambda_{\Theta,v}([n]P)\geq (1/2-1/M)^{2}\pi\Big(\Tr(\Ima
\tau_{v})-2\Ima\tau_{v,12}\Big)-\log\Big(4+\frac{3}{2}\Tr(\Ima \tau_{v})\Big)+\log\frac{M}{240}.$$
\end{proof}

\subsection{Minoration globale : preuve du th\'eor\`eme \ref{minoration dimension 2}}\label{preuve du thm}

\begin{proof}

Soit $k'=k(A[3])$. Posons $m'=|M_{k'}^{\infty}|$ et $d'=[k':\mathbb{Q}]$. Prenons $M>2$ un param\`
etre r\'eel \`a fixer ult\'erieurement. \'Ecrivons l'\'egalit\'e-clef \ref{clef} sur $k'$ en choisissant pour $n$ l'entier donn\'e par le principe des tiroirs de \ref{minoration infinie} :
$$n^2\widehat{h}_{A,2\Theta}(P)=\sum_{v\in{M_{k}^{0}}}n_{v}\widehat{\lambda}_{2\Theta,v}([n]P)+\sum_{v\in{M_{k}^{\infty}}}n_{v}2\Lambda_{\Theta,v}([n]P)+\frac{3}{20d}\log\Nk(D)+\frac{1}{10}\sum_{v\in{M_{k}^{\infty}}}n_{v}\log|\Delta(\tau_{v})|_{v}.$$

Il suffit alors d'appliquer les minorations locales des propositions \ref{minoration finie},
\ref{minoration delta} et \ref{minoration infinie} pour obtenir :

\begin{tabular}{lll}
$\displaystyle{n^2\widehat{h}_{A,2\Theta}(P)}$ & $\geq$ & $\displaystyle{\sum_{v\in{M_{k'}^{0}}}n_{v}\left(\Big(-\frac{1}{3}+\frac{3}{20}\Big)\ordv(D)+\frac{4}{3}\ordv(2)\right)\log \Nk(v) }$\\
\\

&  & $\displaystyle{+\sum_{v\in{M_{k'}^{\infty}}} n_{v}\Big(\Big(2\Big(\frac{1}{2}-\frac{1}{M}\Big)^{2}-\frac{2}{10}\Big)\pi\Tr(\Ima\tau_{v})+\Big(\frac{2}{10}-4\Big(\frac{1}{2}-\frac{1}{M}\Big)^{2}\Big)\pi \Ima\tau_{12}\Big)}$\\
\\

&  & $\displaystyle{+\!\sum_{v\in{M_{k'}^{\infty}}}\! n_{v}\left(-2\log\Big(4\!+\!\frac{3}{2}\Tr(\Ima\tau_{v})\Big)+\frac{1}{10}\log\Big(c_0\min\{1, \pi\vert \tau_{12}\vert_v\}^2 \Big)+2\log\frac{M}{240}\right),}$\\
\\
\end{tabular}

donc en utilisant $\Tr(\Ima\tau)\geq 4\Ima\tau_{12}$ on obtient :
\\

\begin{tabular}{l}
$\displaystyle{n^2\widehat{h}_{A,2\Theta}(P)}\geq\displaystyle{-\frac{11}{60d'}\log \Nk(D)}+\sum_{v\in{M_{k'}^{\infty}}} n_{v}\Big(\Big(\Big(\frac{1}{2}-\frac{1}{M}\Big)^{2}-\frac{2}{10}\Big)\pi\Tr(\Ima\tau_{v})+\frac{2}{10}\pi \Ima\tau_{12}\Big)$\\
\\

$\displaystyle{+\sum_{v\in{M_{k'}^{\infty}}} n_{v}\left(-\!2\log\Big(4+\frac{3}{2}\Tr(\Ima\tau_{v})\Big)\!+\frac{1}{10}\log\Big(c_0\min\{1, \pi\vert \tau_{12}\vert_v\}^2\Big)\right)+2\log\frac{M2^{2/3}}{240}},$\\
\\
\end{tabular}

donc comme $\Ima\tau_{12}\geq \vert\tau_{12}\vert-\frac{1}{2}\geq \log\vert\tau_{12}\vert$ :

\vspace{0.3cm}

\begin{tabular}{lll}
$\displaystyle{n^2\widehat{h}_{A,2\Theta}(P)}$ & $\geq$ & $\displaystyle{\sum_{v\in{M_{k'}^{\infty}}} n_{v}\left[\Big(\Big(\frac{1}{2}-\frac{1}{M}\Big)^{2}-\frac{1}{5}\Big)\pi\Tr(\Ima\tau_{v})-\!2\log\Big(4+\frac{3}{2}\Tr(\Ima\tau_{v})\Big)\right]}$\\
\\

&  & $\displaystyle{-\frac{11}{60d'}\log \Nk(D)+\frac{1}{5d'}\log\prod_{v\in{M_{k'}^{\infty}}} \vert\tau_{12}\vert_v^{d_{v}}+2\log\frac{Mc_0^{1/20}\pi^{1/5}2^{2/3}}{240}.}$\\
\\

\end{tabular}

Prenons \`a pr\'esent pour $M$ la partie enti\`ere sup\'erieure de $240c_0^{-1/20}\pi^{-1/5}2^{-2/3}\sqrt{1044}$, ce qui num\'eriquement donne $M=10087$. 

Un calcul de variation fournit alors $$\Big(\Big(\frac{1}{2}-\frac{1}{M}\Big)^{2}-\frac{1}{5}\Big)\pi\Tr(\Ima\tau_{v})-\!2\log\Big(4+\frac{3}{2}\Tr(\Ima\tau_{v})\Big)+6,95\geq 0,12 \Tr(\Ima\tau_v).$$

On tient alors compte de $n\leq 2\cdot10087^{4d'}$. On conclut cette preuve en redescendant sur le corps de base. On sait que la vari\'et\'e est d\'efinie sur $k$. De plus les quantit\'es $\frac{1}{d}\Trinf(A)$, $\frac{1}{d} s_{\infty}(A)$ et $\frac{1}{d}\log\Nk(D)$ sont invariantes par extension de corps : pour la trace archim\'edienne et la simplicit\'e archim\'edienne c'est montr\'e dans le corollaire \ref{bonne r\'eduction potentielle}, pour le discriminant c'est une cons\'equence directe de la multiplicativit\'e des normes et de la multiplicativit\'e des degr\'es. On a donc :
$$\frac{1}{d'}\Big(\Trinf(A)-\frac{5}{3}\log\frac{\Nkprime(D)}{\sinf(A)}\Big)=\frac{1}{d}\Big(\Trinf(A)-\frac{5}{3}\log\frac{\Nk(D)}{\sinf(A)}\Big).$$
Enfin par multiplicativit\'e des degr\'es \`a nouveau : $d'=[k':k]d\leq 3^{16}d$. Pour la version B du th\'eor\`eme, il suffit de prendre $M=10087\, \Nkprime(D)^{5/6d'}\,s_\infty(A)^{-5/6d'}$.
\end{proof}

\section{La hauteur de Faltings}

\subsection{Expression dans le mod\`ele d'Igusa}

Soit $k$ un corps de nombres. Soient $C/k$ une coube lisse de genre $2$
et $A=\Jac(C)$ sa jacobienne. Notons $\hF(A/k)$ la hauteur de Faltings de la vari\'et\'e ab\'elienne $A/k$. On suppose de plus que la
courbe localis\'ee $C_{p}$  est lisse de genre 2 en toute place $p$
divisant $2$. On note $\Delta_{\mathrm{min}}$ le discriminant minimal associ\'e aux
mod\`eles d'Igusa de la courbe $C$, lequel est utilis\'e dans \cite{Ueno}
et d\'efini dans le paragraphe suivant. On notera $\mathcal{Z}_{2}$ l'ensemble des caract\'eristiques paires de dimension 2.
En se r\'ef\'erant aux travaux de K. Ueno de l'article \cite{Ueno} page 765 on a :
$$\hF(A/k)=\frac{1}{10d}\left[\sum_{p\in{M_{k}^{0}}}d_{p}\ordp(2^{-12}\Delta_{\mathrm{min}})\log \Nk(p)\! -\! \sum_{v\in{M_{k}^{\infty}}}d_{v}\log \Big|\Delta(\tau_v)\det(\Ima \tau_{v})^{5}\Big|\right].$$ \label{faltings}

La formule obtenue pour la hauteur de Faltings modifi\'ee est donc :
$$\hFprime(A/k)=\frac{1}{10d}\left[\sum_{p\in{M_{k}^{0}}}d_{p}\ordp(2^{-12}\Delta_{\mathrm{min}})\log \Nk(p) - \sum_{v\in{M_{k}^{\infty}}}d_{v}\log \Big|\Delta(\tau_v)\Big|\right].$$

\subsection{Comparaison entre discriminants}

\begin{lem} \label{comparaison disc}
Soit $C/k$ une courbe de genre 2 donn\'ee dans le mod\`ele $y^{2}=F(x)$ hyperelliptique entier sur $k$ (comme dans l'article \cite{FlySma}), avec bonne r\'eduction en toute place divisant $2$. Le discriminant $\Delta_{\mathrm{min}}$ introduit dans l'article \cite{Ueno} du mod\`ele d'Igusa de $C$ v\'erifie :
$$\sum_{p\in{M_{k}^{0}}}d_{p}\ordp(2^{-12}\Delta_{\mathrm{min}})\log \Nk(p) \leq \sum_{p\in{M_{k}^{0}}}d_{p}\ordp(2^{8}\disc(F))\log \Nk(p).$$
\end{lem}

\begin{proof}
Le mod\`ele d'Igusa est donn\'e par une \'equation du type :
$$ xy^{2}+(1+ax+bx^{2})y+x^{2}(c+dx+x^{2})=0. $$
Son discriminant est d\'efini dans \cite{Ueno} comme \'etant le discriminant de l'\'equation hyperelliptique :
$$ y^{2}=(1+ax+bx^{2})^{2}-4x^{3}(c+dx+x^{2}),$$
corrig\'e par une puissance de $2$, afin de tenir compte du comportement aux places de $k$ divisant $2$.
Le discriminant minimal donn\'e dans \cite{Ueno} est donc de norme inf\'erieure ou \'egale au discriminant minimal de la courbe hyperelliptique $C$ (car il est plus petit pour la valuation en $2$). Ce discriminant sera en particulier de norme inf\'erieure ou \'egale \`a celle du discriminant du mod\`ele hyperelliptique de Flynn-Smart (qui n'est pas forc\'ement le produit des discriminants minimaux locaux), on consultera par exemple \cite{Liu2} page 4581 et suivantes.
\end{proof}

\subsection{Majoration de $\hFprime(A/k)$ : preuve du th\'eor\`eme \ref{faltings maj}}

On montre dans ce paragraphe une majoration de la hauteur de Faltings des surfaces ab\'eliennes simples. La pr\'esence de la quantit\'e $\Delta(\tau)$ aux places archim\'ediennes, quantit\'e qui s'annule en $\tau_{12}=0$, impose la condition $\tau_{12}\neq 0$.

\begin{proof}
Pour le terme non archim\'edien on a en utilisant le lemme \ref{comparaison disc} :
$$\sum_{p\in{M_{k}^{0}}}d_{p}\ordp(2^{-12}\Delta_{\mathrm{min}})\log \Nk(p) \leq \sum_{p\in{M_{k}^{0}}}d_{p}\ordp(2^{8}\disc(F))\log \Nk(p),$$
donc en utilisant la proposition \ref{minoration delta} :

\begin{center}
\begin{tabular}{lll}
$10d\hFprime(A/k)$ & $\leq$ & $\displaystyle{\log\Nk(D)+\sum_{v\in{M_{k}^{\infty}}}d_{v}\Big(2\pi(\Ima \tau_{v,1}+\Ima \tau_{v,2}-\Ima \tau_{v,12})\Big)}$\\
\\
&  & $\displaystyle{-\sum_{v\in{M_{k}^{\infty}}}d_{v}\log\Big(c_0\min\{1, \pi\vert \tau_{12}\vert_v\}^2\Big)}.$\\
\end{tabular}
\end{center}

En utilisant l'estimation $c_0=5.10^{-9}$ de la proposition \ref{minoration delta} et les in\'egalit\'es $\Ima\tau_{12,v}\geq \log\vert\tau_{12}\vert_v$ et $\Tr(\Ima\tau_v)\geq \sqrt{3}$ il vient alors :
$$\hFprime(A/k)\leq \frac{1}{10d}\log\frac{\Nk(D)}{s_\infty(A)}+\frac{1}{10}\sum_{v\in{M_{k}^{\infty}}}\!n_{v}\Big(6\pi(\Ima \tau_{v,1}+\Ima \tau_{v,2})\Big).$$
\end{proof}

\section{Produit de courbes elliptiques}

On donne dans ce paragraphe un th\'eor\`eme \'equivalent pour le cas des produits de courbes elliptiques. Ce r\'esultat est directement construit \`a partir des r\'ef\'erences \cite{Sil2} et \cite{HiSi3} et est \'ecrit en d\'etails dans \cite{Paz}. Ce th\'eor\`eme est plus faible que le r\'esultat de M. Hindry et J. Silverman de \cite{HiSi3} mais permet d'obtenir un \'enonc\'e plus homog\`ene pour les vari\'et\'es ab\'eliennes de dimension 2. Dans toute cette partie, les matrices de p\'eriodes $\tau_v$ sont dans le domaine fondamental usuel.

\begin{thm} \label{Neron-Tate elliptique}

Soit $k$ un corps de nombres de degr\'e $d$. On note $m=|M_{k}^{\infty}|$. Alors il existe une constante $c_{1}(d)>0$ telle que pour toute courbe elliptique $E/k$ de discriminant minimal $\Delta_{E}$ et de trace archim\'edienne $\Trinf(E)$, et pour tout point $P\in{E(k)}$ d'ordre infini :
$$\widehat{h}_{E}(P)\geq c_{1}\,\Big(\Trinf(E) - \frac{1}{7,2}\log\Nk(\Delta_{E})\Big)\,,$$
o\`u on peut prendre $c_{1}=0,3\cdot \left(d\,20^{4m}\right)^{-1}$.
\end{thm}

\begin{Remarque} On va \^etre amen\'e dans la suite du texte \`a imposer $\Trinf(E)\geq\frac{1}{7}\log\Nk(\Delta_{E})$. \'Etudions cette condition sur $\mathbb{Q}$. Fixons une courbe elliptique $E/\mathbb{Q}$ et supposons $|j(E)|\gg 1$. Alors :
$$\Trinf(E)=\Ima\tau_{E}=-\frac{1}{2\pi}\log|q_{E}|\simeq \frac{1}{2\pi}\log|j(E)|.$$
Donc si $\vert j(E)\vert$ est grand, une hypoth\`ese du type $\Trinf(E)\gg \frac{1}{7}\log|\Delta_{E}|$ \'equivaut \`a $|j(E)|\gg |\Delta_{E}|^{\frac{2\pi}{7}}$. Or on a la relation $j(E)=1728 c_{4}^{3}/\Delta_{E}$, o\`u $c_{4}$ est un polynôme en les coefficients de la courbe elliptique (voir \cite{Sil1} page 46). On a donc :
$$|c_{4}|^{3}\gg |\Delta_{E}|^{1+\frac{2\pi}{7}}.$$
La conjecture de Hall (voir \cite{Sil1} page 268) donne l'in\'egalit\'e :
$$|\Delta_{E}|\gg |c_{4}|^{\frac{1}{2}-\varepsilon}.$$
Comme $3>\frac{1}{2}+\frac{\pi}{7}$, ces in\'egalit\'es sont compatibles, mais il est important de garder \`a l'esprit que la constante de comparaison entre la trace archim\'edienne et le logarithme du discriminant ne saurait \^etre trop grande dans le cas de la dimension 1.

Si on s'autorise des extensions de corps, on peut bien entendu raisonner comme dans la remarque \ref{remarque dim2} pour obtenir des familles de courbes elliptiques avec $\Trinf(E)/\log\vert\Nk(\Delta_E)\vert$ grand.
\end{Remarque}

\subsection{Majoration de la hauteur de Faltings}

Nous allons montrer dans cette partie la majoration suivante :

\begin{thm} \label{Faltings elliptique}

Soit $E/k$ une courbe elliptique donn\'ee dans un mod\`ele entier de Weierstrass
de discriminant $\Delta_{E}$ et de trace archim\'edienne $\Trinf(E)$. Alors on a :
$$\hFprime(E/k)\leq
c_{3}\Trinf(E)+c_{4}\log\Nk(\Delta_{E}),$$
o\`u on peut prendre $c_{3}=\frac{2\pi}{12d}$ et $c_{4}=\frac{1}{12d}$.
\end{thm}

\begin{proof}
On conna\^it une expression explicite de la hauteur de Faltings d'une
courbe elliptique $E$ d\'efinie sur un corps de nombres $k$ (voir par exemple \cite{CorSil} page 254) :
$$\hF(E/k)=\frac{1}{12d}\left(\log\Nk(\Delta_{E})-\sum_{v\in{M_{k}^{\infty}}}d_{v}\log\vert\Delta(\tau_{v})(\Ima\tau_{v})^{6}\vert
\right).$$
Ceci donne :
$$\hFprime(E/k)=\frac{1}{12d}\left(\log\Nk(\Delta_{E})-\sum_{v\in{M_{k}^{\infty}}}d_{v}\log\vert\Delta(\tau_{v})\vert
\right).$$
Partant de cette expression il suffit donc de relier la fonction
$\Delta(\tau_{v})$ \`a la quantit\'e $\Ima\tau_{v}$. On note
$q=\exp(2i\pi\tau_{v})$ ; on a la formule :
$$\Delta(\tau_{v})=(2\pi)^{12}q\prod_{n=1}^{+\infty}(1-q^{n})^{24},
$$
donc :
$$-\log|\Delta(\tau_{v})|\leq
2\pi\Ima\tau_{v}-12\log2\pi+24C_{\tau_{v}}, $$
avec :
$$C_{\tau_{v}}= -\sum_{n=1}^{+\infty}\log|1-q^{n}|\leq
-\sum_{n=1}^{+\infty}\log(1-e^{-2\pi\Ima\tau_{v} n})\leq \sum_{n=1}^{+\infty}2e^{-2\pi\Ima\tau_{v}n}, $$
o\`u on a utilis\'e, pour $x\leq1/2$ l'in\'egalit\'e $-\log(1-x)\leq 2x$. Donc :
$$C_{\tau_{v}}\leq
2\sum_{n=1}^{+\infty}e^{-2\pi\Ima\tau_{v}n}=2\frac{e^{-2\pi\Ima\tau_{v}}}{1-e^{-2\pi\Ima\tau_{v}}}\leq
2\frac{e^{-\sqrt{3}\pi}}{1-e^{-\sqrt{3}\pi}}\leq 0,01.$$
Il suffit d'injecter cette majoration dans l'expression de la hauteur
de Faltings et d'utiliser $2\pi\Ima\tau_{v}-12\log2\pi+0,24\leq 2\pi\Ima\tau_{v}$
pour conclure.
\end{proof}

\section{Corollaires}

On pr\'esente dans cette partie plusieurs \'enonc\'es. Les premiers corollaires constituent une avanc\'ee en direction de la conjecture de Lang et Silverman en dimension 2. Par la suite on pr\'esente une borne uniforme explicite pour la torsion d'une famille de vari\'et\'es ab\'eliennes de dimension 2. Enfin on obtient une borne explicite sur le nombre de points rationnels pour des familles de courbes de genre 2.

\subsection{Conjecture de Lang et Silverman en dimension 2}

Soit $(A,\mathcal{D})$ une vari\'et\'e ab\'elienne principalement polaris\'ee de dimension 2. Comme expliqu\'e dans l'introduction, il y a alors deux possibilit\'es (on pourra aussi consulter \cite{Weil}) :  

\begin{center} \;
$\left\{ 
\begin{tabular}{l}
$(A,\mathcal{D})\simeq (E_{1}\times E_{2},\;E_{1}\times \{O\}+\{O\}\times E_{2}),$ \\ 
$\;\;\mathrm{ou}$ \\ 
$(A,\mathcal{D})\simeq (\Jac(C),\Theta),$ \\

\end{tabular}
\right.$
\end{center}

o\`u $C$ est une courbe alg\'ebrique de genre 2. Dans le premier cas, en notant $\mathcal{L}=E_{1}\times \{O\}+\{O\}\times E_{2}$, on a les relations :

\begin{center} \;
$\left\{ 
\begin{tabular}{l}
$\widehat{h}_{E_{1}\times E_{2},\mathcal{L}}((P_{1},P_{2}))=\widehat{h}_{E_{1}}(P_{1})+\widehat{h}_{E_{2}}(P_{2}),$ \\ 
 \\ 
$\hF(E_{1}\times E_{2}/k)=\hF(E_{1}/k)+\hF(E_{2}/k).$ \\

\end{tabular}
\right.$
\end{center}

\begin{Remarque}
On peut \'eventuellement translater le diviseur $(E_{1}\times \{O\}+\{O\}\times E_{2})$ par un point $Q$ de $2$-torsion, ce qui ne change pas le calcul en vertu du fait que $\widehat{h}_{E}(P+Q)=\widehat{h}_{E}(P)$ pour tout point $P$ de $E$.
\end{Remarque}

Ceci permet donc d'utiliser les th\'eor\`emes \ref{Neron-Tate elliptique}
et \ref{Faltings elliptique} pour obtenir un \'enonc\'e dans la direction de la conjecture de Lang et Silverman :

\begin{cor} \label{elliptique+elliptique}

Soient $E_{1}/k$ et $E_{2}/k$ deux courbes elliptiques. On consid\`ere
la vari\'et\'e ab\'elienne $E_{1}\!\times\! E_{2}$ munie de la polarisation
$E_{1}\!\times\! \{O\}+\{O\}\!\times\! E_{2}$. On pose pour $i=1,2$:
$$\Trinf(E_{i})=\sum_{v\in{M_{k}^{\infty}}}d_{v}\Ima\tau_{v}^{(i)}.
$$
On suppose que $\Trinf(E_{i})\geq\frac{1}{7}\log\Nk(\Delta_{E_{i}})$. Alors pour tout  $P_{1}\in{E_{1}(k)}$ et $
P_{2}\in{E_{2}(k)}$ points d'ordre infini :
\[
\widehat{h}_{E_{1}\!\times\! E_{2}}(P_{1},P_{2})\geq c_{0}\, \hF(E_{1}\!\times\! E_{2}/k),
\]
o\`u on peut prendre $c_{0}=0,0025\cdot20^{-4m}$.

\end{cor}

\begin{proof}

En utilisant les th\'eor\`emes \ref{Neron-Tate elliptique} et \ref{Faltings
  elliptique} on a les estimations pour chacune des courbes $E_{i}$ avec $i\in{\{1,2\}}$ :
$$\widehat{h}_{E_{i}}(P)\geq c_{1}\Trinf(E_{i}) - c_{2}\log\Nk(\Delta_{E_{i}}),$$
$$\hF(E_{i}/k)\leq c_{3}\Trinf(E_{i})+c_{4}\log\Nk(\Delta_{E_{i}}), $$
o\`u on peut prendre :

\begin{center}
\begin{tabular}{llll}

$c_{1}=\frac{0.3}{d\,20^{4m}},$ &
$c_{2}=\frac{1}{24d\,20^{4m}},$ & $c_{3}=\frac{2\pi}{12d},$ &  $c_{4}=\frac{1}{12d}.$\\

\end{tabular}
\end{center}

En utilisant de plus l'hypoth\`ese $\Trinf(E_{i})\geq\frac{1}{7}\log\Nk(\Delta_{E_{i}})$ il vient :
$$\widehat{h}_{E_{1}\!\times\!E_{2}}(P_{1},P_{2})=\widehat{h}_{E_{1}}(P_{1})+\widehat{h}_{E_{2}}(P_{2})\geq c_{0}\, \hF(E_{1}/k)+c_{0}\, \hF(E_{2}/k)=c_{0}\, \hF(E_{1}\!\times\! E_{2}/k), $$
o\`u on a not\'e :
$$c_{0}=\left(c_{1}-\frac{c_{2}}{1/7}
\right)\left(c_{3}+\frac{c_{4}}{1/7} \right)^{-1}=\frac{c_{1}-7c_{2}}{c_{3}+7c_{4}}.$$

\end{proof}

Il reste donc \`a \'etudier le cas des jacobiennes de
courbes de genre 2. Or nous sommes \`a pr\'esent en mesure de construire un \'enonc\'e de th\'eor\`eme r\'epondant partiellement \`a la conjecture de Lang et Silverman pour ces vari\'et\'es ab\'eliennes particuli\`eres. 

En r\'eunissant les r\'esultats des th\'eor\`emes \ref{minoration dimension 2} et \ref{faltings maj} on obtient une preuve du corollaire \ref{genre 2}, en consid\'erant toujours $D=2^{8}\disc(F)$ si $C:y^{2}=F(x)$ avec $\deg(F)=5$ et $\Trinf(A)$ la trace archim\'edienne de $A$ :

\begin{proof}
 En utilisant les th\'eor\`emes \ref{minoration dimension 2} et \ref{faltings maj} on a les estimations : 
$$\widehat{h}_{A,2\Theta}(P)\geq c_1\left(\Trinf(A)-\frac{5}{3}\log\frac{\Nk(D)}{\sinf(A)}\right),$$
et :
$$\hFprime(A/k)\leq c_3\Trinf(A) + c_4\log\frac{\Nk(D)}{\sinf(A)}.$$

En notant :

\begin{center}
\begin{tabular}{llll}

$c_{1}=\frac{0,03}{d\,10087^{8\cdot3^{16}d}}$  & $c_2=\frac{5}{3}c_1$ &
$c_{3}=\frac{6\pi}{10\,d}$ & $c_{4}=\frac{1}{10\,d}$\\

\end{tabular}
\end{center}

et en supposant : $\Trinf(A)\geq (5/3 +\varepsilon)\log \Nk(D)$, on obtient alors :
$$\widehat{h}_{A,2\Theta}(P)\geq \left(c_{1}-\frac{c_2}{5/3+\varepsilon}\right)\left(c_{3}+\frac{c_{4}}{5/3+\varepsilon} \right)^{-1} \hFprime(A/k).  $$
\end{proof}

On d\'eduit de ces \'enonc\'es le corollaire suivant :

\begin{cor}\label{langdim2}

Soit $k$ un corps de nombres de degr\'e $d$. Alors il existe une constante $c=c(d)>0$ ne d\'ependant que du degr\'e de $k$ telle que pour toute vari\'et\'e ab\'elienne $(A,\Theta)$ sur $k$, principalement polaris\'ee
de dimension 2, v\'erifiant les hypoth\`eses des \'enonc\'es
\ref{elliptique+elliptique} ou \ref{genre 2} et pour tout
point $P\in{A(k)}$ tel que $\mathbb{Z}\!\cdot\! P$ est
Zariski-dense on a :
$$\widehat{h}_{A,\Theta}(P)\geq c \, \hFprime(A/k), $$
et on peut prendre $c=\min\{c_{0},c_{1}\}=c_{1}$, avec $c_{0}$ et
$c_{1}$ les constantes donn\'ees respectivement dans les \'enonc\'es
 \ref{elliptique+elliptique} et \ref{genre 2}. 
\end{cor}

\begin{Remarque}
On obtient la conjecture de Lang et Silverman (sous les hypoth\`eses des \'enonc\'es utilis\'es) en remarquant que $\hFprime(A/k)\geq \hF(A/k)$ est valable lorsque $\Ima\tau$ est suffisament grand.
\end{Remarque}

On peut de plus d\'eduire de \ref{genre 2} une preuve du corollaire \ref{bonne r\'eduction potentielle} :

\begin{proof} On sait que si le mod\`ele de
la courbe est \`a bonne r\'eduction partout et est globalement minimal on obtient $\Nk(D)=1$.
Or il existe une extension $k'$ de $k$ telle que $A/k'$ est \`a bonne r\'eduction partout. On va voir qu'en faisant une autre extension bien choisie on peut de plus obtenir l'existence d'un mod\`ele globalement minimal : donnons-nous tout d'abord un mod\`ele hyperelliptique entier sur $\mathcal{O}_{k}$ de $C$, dont le discriminant sera not\'e $\Delta_{C}$. En se reportant par exemple \`a \cite{Lock} page 736, on sait que pour toute place finie $v$, il existe un entier $u_{v}$ tel que $\Delta_{C}=u_{v}^{40}\Delta_{v}$, o\`u $\Delta_{v}$ est le discriminant minimal local. L'exposant $40$ vient du fait qu'on est ici en dimension $g=2$, et $4g(2g+1)=40$ dans ce cas.

On pose alors $\displaystyle{\mathfrak{a}_{C}:= \prod_{v}\mathfrak{p}_{v}^{-\ordv(u_{v})}}$. On obtient facilement les faits suivants (voir \cite{Lock}) : $\Delta_{\mathrm{min}}=\Delta_{C}(\mathfrak{a}_{C})^{40}$, o\`u $\Delta_{\mathrm{min}}$ est le discriminant minimal de la courbe hyperelliptique $C$. De plus la classe d'id\'eaux de $\mathfrak{a}_{C}$ ne d\'epend pas du mod\`ele hyperelliptique de $C$. Enfin il existe un mod\`ele minimal global si et seulement si $\mathfrak{a}_{C}$ est principal.

Or sur $k'$, on a bonne r\'eduction partout, ce qui impose $\Delta_{\mathrm{min}}\mathcal{O}_{k'}=\mathcal{O}_{k'}$. En particulier on obtient que l'id\'eal $\mathfrak{a}_{C}^{40}$ est principal sur $k'$. Il existe donc $\alpha\in{k'}$ tel que $\mathfrak{a}_{C}^{40}=\alpha \mathcal{O}_{k'}$. Consid\'erons alors $k''=k'[\beta]$, avec $\beta^{40}=\alpha$. Alors $\mathfrak{a}_{C}=\beta \mathcal{O}_{k''}$ est principal sur $k''$, et le degr\'e de l'extension $[k'':k']$ est inf\'erieur ou \'egal \`a $40$.

La vari\'et\'e ab\'elienne $A$ \'etant d\'efinie sur $k$, elle l'est aussi sur $k'$ et $k''$. De plus on a les relations :
$$\forall w\in{M_{k''}^{\infty}},\,\; w|v \Rightarrow \;
\tau_{w}=\tau_{v}.$$
On peut donc appliquer les th\'eor\`emes \ref{minoration dimension 2} et
 \ref{faltings maj} \`a $A/k''$ puisque les $\tau_{w}$ v\'erifient les
 m\^emes conditions que les $\tau_{v}$, donc en utilisant de plus $\sinf(A)\geq 1$:
$$\widehat{h}_{A,2\Theta}(P)\geq \frac{1}{10087^{8[k'':\mathbb{Q}]}}\frac{1}{20\pi}\,\hstprime(A).$$

Or $[k'':\mathbb{Q}]=[k'':k'][k':k][k:\mathbb{Q}]$ et on peut d\'eduire de \cite{Sil4} page 400, en choisissant $k'=k[A[15]]$ que $[k':k]\leq 15^{4\!\times\! 4}$. Remarquons qu'il suffit de redescendre sur le corps de base \`a la fin, d'o\`u la pr\'esence du terme $15^{16}$ et non de $3^{16}\!\cdot\!15^{16}$. 
\end{proof}

\begin{Remarque} On aurait pu essayer de se placer sur l'extension $k'$ de
$k$ sur laquelle la vari\'et\'e admet bonne r\'eduction partout, puis monter
jusqu'\`a $H_{k'}$ le corps de classes de Hilbert de $k'$ sur lequel le
mod\`ele est globalement minimal (par principalit\'e) et a toujours bonne
r\'eduction partout. Cependant la constante obtenue d\'ependra alors du
corps $k'$ aussi.
\end{Remarque}

\subsection{Borne pour la torsion d'une jacobienne de dimension 2.}

Le principe des tiroirs utilis\'e dans la preuve du th\'eor\`eme \ref{minoration
  infinie} montre le fait suivant : si on peut obtenir suffisamment de
multiples distincts d'un point $P$, alors la hauteur de N\'eron-Tate de ce point
est minor\'ee par une quantit\'e non nulle, donc ce point n'est pas un
point de torsion. Inversement on va donc obtenir une borne sur la
torsion des jacobiennes sur lesquelles on a travaill\'e dans le th\'eor\`eme
\ref{minoration dimension 2}. Il suffit d'\'elever la borne sur l'exposant du groupe \`a la puissance $2g=4$ pour obtenir la preuve du corollaire \ref{borne torsion}.

\subsection{Borne pour les points rationnels d'une courbe de genre 2.}
L'obtention d'un r\'esultat de minoration du type Lang-Silverman sur une famille de jacobiennes donne syst\'ematiquement un majorant du nombre de points rationnels des courbes sous-jacentes. Le calcul de ce majorant en fonction de la constante de l'in\'egalit\'e de Lang-Silverman est montr\'e dans \cite{Paz2}, Proposition 1.10. Ainsi, pour obtenir une preuve du corollaire \ref{points rationnels}, il suffit d'appliquer la Proposition 1.10 de \cite{Paz2}, avec ici $g=2$. Ces questions ont \'et\'e abord\'ees par G. R\'emond, voir la proposition 3.7 page 527 de l'article \cite{Rem2} ainsi que les estimations de \cite{DavPhi} (page 652, page 662 et page 665) et T. de Diego, voir par exemple \cite{DeDiego} page 109.

\bibliographystyle{amsalpha}

\vfill
{\flushright
Fabien Pazuki\\
Th{\'e}orie des nombres, IMB Universit{\'e} Bordeaux 1\\
351, cours de la Lib\'eration, 33405 Talence cedex, France\\
e-mail : fabien.pazuki@math.u-bordeaux1.fr\\
}

\end{document}